\documentclass[12pt]{article}
\usepackage{amsthm}
\usepackage{times}
\usepackage{bbding}
\usepackage{authblk}
\usepackage[margin=1in]{geometry}
\usepackage[utf8]{inputenc}
\usepackage{amsmath,accents}
\usepackage[normalem]{ulem}
\usepackage{mathtools}
\usepackage{amsfonts}
\usepackage{graphicx}
\usepackage{amssymb}
\usepackage{xcolor}
\usepackage{algorithm}
\usepackage{bbm}
\usepackage{relsize}
\usepackage{tikz-cd}
\usepackage[colorlinks]{hyperref}
\usepackage{algorithmic}
\usepackage{MnSymbol}
\usepackage{cleveref}
\newtheorem{theorem}{Theorem}[section]
\newtheorem{remark}[theorem]{Remark}
\newtheorem{lemma}[theorem]{Lemma}
\newtheorem{corollary}[theorem]{Corollary}
\newtheorem{proposition}[theorem]{Proposition}
\newtheorem{definition}[theorem]{Definition}

\newtheorem{example}[theorem]{Example}

\newtheorem{claim}[theorem]{Claim}

\newcommand{\rank}{\mathrm{rank}}
\newcommand{\T}{\mathrm{T}}
\newcommand{\lap}{\pmb{\Delta}}
\newcommand{\eps}{\varepsilon}
\newcommand{\R}{\mathbb{R}}

\newcommand{\lc}{\left(}
\newcommand{\rc}{\right)}

\title{Persistent Laplacians: properties, algorithms and implications} 

\author[1]{Facundo M\'emoli\thanks{\url{memoli@math.osu.edu}}}
\author[2]{Zhengchao Wan\thanks{\url{wan.252@osu.edu}}}
\author[3]{Yusu Wang\thanks{\url{yusuwang@ucsd.edu}}}

\affil[1]{Department of Mathematics and Department of Computer Science and Engineering, The Ohio State University, Columbus, OH, USA}

\affil[2]{Department of Mathematics, The Ohio State University, Columbus, OH, USA}
 	
\affil[3]{Hal{\i}c{\i}o\u{g}lu Data Science Institute, University of California San Diego, San Diego, CA, USA}

\date{}

\begin{document}

\maketitle

\begin{abstract}
We present a thorough study of the theoretical properties and devise efficient algorithms for the \emph{persistent Laplacian}, an extension of the standard combinatorial Laplacian to the setting of pairs (or, in more generality, sequences) of simplicial complexes $K \hookrightarrow L$,  which was independently introduced by Lieutier et al. and by Wang et al. In particular, in analogy with the non-persistent case, we first prove that the nullity of the $q$-th persistent Laplacian $\Delta_q^{K,L}$ equals the $q$-th persistent Betti number of the inclusion $(K \hookrightarrow L)$. We then present an initial algorithm for finding a matrix representation of $\Delta_q^{K,L}$, which itself helps interpret the persistent Laplacian.  We exhibit a novel relationship between the persistent Laplacian and the notion of Schur complement of a matrix which has several important implications. In the graph case, it both uncovers a link with the notion of effective resistance and leads to a persistent version of the Cheeger inequality. This relationship also yields an additional, very simple algorithm for finding (a matrix representation of) the $q$-th persistent Laplacian which in turn leads to a novel and fundamentally different algorithm for computing the $q$-th persistent Betti number for a pair $(K,L)$ which can be significantly more efficient than standard algorithms. 
Finally, we study persistent Laplacians for simplicial filtrations and present novel stability results for their eigenvalues. Our work brings methods from spectral graph theory, circuit theory, and persistent homology together with a topological view of the combinatorial Laplacian on simplicial complexes.\end{abstract}

\newpage

\tableofcontents

\section{Introduction} 

The combinatorial graph Laplacian, as an operator on functions defined on the vertex set of a graph, is a fundamental object in the analysis of and optimization on graphs. 
Its spectral properties are  widely used in graph optimization problems (e.g, spectral clustering \cite{chung1997sgt,LeeGT12,ng2002sca,uvon}) and in the efficient solution of systems of equations, cf. \cite{KMP12,livne2012lean,spielman2004nearly,Vishnoi13}. The graph Laplacian  is also connected to network circuit theory via the notion of \emph{effective resistance} \cite{Bollobas98,dorfler2012kron,lyons2017probability,spielman2011graph}. 

There is also an algebraic topology view of the graph Laplacian which arises through considering boundary operators and specific inner products defined on simplicial (co)chain groups \cite{chung1997sgt}.
This permits extending the graph Laplacian to a more general operator, the $q$-th combinatorial Laplacian $\Delta_q^K$ on {the $q$-th (co)chain groups of} a given simplicial complex $K$ (see e.g., \cite{eckmann1944harmonische,duval2002shifted,goldberg2002combinatorial,horak2013spectra}), so that the standard graph Laplacian simply corresponds to the 0-th case. 
These ideas connect to the topology of the input simplicial complex via the so called \emph{combinatorial Hodge Theorem} \cite{eckmann1944harmonische}, which states that the nullity of the $q$-th combinatorial Laplacian is equal to the rank of the $q$-th cohomology group of $K$ with real coefficients, i.e. the $q$-th Betti number of $K$. See also \cite{horak2013spectra,lim2020hodge} for thorough expositions.

The combinatorial Laplacian (and  variants) have received a great deal of attention in recent years; see e.g. \cite{goldberg2002combinatorial,GundertS14,gundert2015higher,MW09}. For example,  \cite{kook2018simplicial}  aims to extend the related concept, effective resistance from network circuit theory, to this ``high dimensional" situation, whereas \cite{hansen2019toward,hansen2020laplacians} consider a spectral theory of cellular sheaves with applications to sparsification and synchronization problems.

Adopting the algebraic topology view of the $q$-th combinatorial Laplacian, \cite{lapslides} and \cite{wang2020persistent} independently introduced the so-called $q$-th \emph{persistent Laplacian $\Delta_q^{K,L}$}, which is an extension of the combinatorial Laplacian mentioned above  to \emph{a pair} of simplicial complexes $K \hookrightarrow L$ connected by an inclusion.
To the best of our knowledge, \cite{lapslides} and \cite{wang2020persistent} are the first works which establishes a link between \emph{persistent} homology \cite{edelsbrunner2000topological,zomorodian2005computing}, one of the most important developments in the field of applied and computational topology in the past two decades, with the Laplacian, a common and fundamental object with a vast literature, both in the theoretical and applied domains. See also \cite{de2011persistent,perea2018multiscale} for other work in computational topology which leverages ideas connected to the (standard) combinatorial Laplacian.

It is thus natural and also highly desirable to achieve better understanding, as well as  algorithmic developments, for this persistent Laplacian,  all of which will  help broaden its potential applications. The present paper aims to close this gap. 

\paragraph{Contributions} In this paper, we carry out a thorough study of the properties of and develop algorithms for the persistent Laplacian. Our work brings together ideas and methods from several communities, including spectral graph theory, circuit theory, topological treatments of high-dimensional combinatorial Laplacians, together with a persistent homology perspective (both at the theoretical and algorithmic levels). For instance, we relate the computation of persistent homology with notions from network theory such as the Kron reduction (and also Schur complements) which have novel algorithmic implications; see below.

This is an overview of our results: 
\begin{itemize}
    \item In \Cref{sec:pairs}, we present several 
    results about the 
    properties of the $q$-th persistent Laplacian $\Delta_q^{K,L}$, including \Cref{thm:pers-betti-pers-lap}, which establishes that the nullity of 
    $\Delta_q^{K,L}$ equals the $q$-th \emph{persistent} Betti number from $K$ to $L$: a result analogous to the one that holds in the non-persistent case. 
    \item In \Cref{sec:first algo persistent Laplacian}, we give a first algorithm (\Cref{algo-pers-lap}) to compute a matrix representation 
    of $\Delta_q^{K,L}$, which relies on  matrix reduction ideas which are standard when computing persistent homology. 
    
    \item In \Cref{sec:schur complement}, we establish our main observation \Cref{thm:persis-Laplacian-schur-formula}, a relationship between the persistent Laplacian and the concept of \emph{Schur complement} of a matrix. This observation has several immediate and important implications:
    \begin{enumerate}
        \item We establish a second, very simple algorithm {(\Cref{algo-pers-lap-schur})} which computes the matrix representation of the persistent Laplacian $\Delta_q^{K,L}$ (for any $q$) efficiently, purely based on a linear algebraic formulation (\Cref{thm:persis-Laplacian-schur-formula}). 
        \item This observation leads to a new algorithm to compute the $q$-th persistent Betti number for a pair of spaces in a fundamentally different manner from extant algorithms in the computational topology literature. This new algorithm is, under mild conditions (e.g. as those commonly satisfied by Vietoris-Rips complexes)   significantly more efficient than existing algorithms. We believe that this new algorithm for computing persistent Betti numbers is of independent interest.
        \item In the graph case (i.e. when $K$ and $L$ are graphs and $q=0$), this provides a direct connection with notions from network circuit theory such as the Kron reduction \cite{dorfler2012kron}, a connection which reveals that the matrix representation of the persistent Laplacian permits recovering the effective resistance of pairs of vertices in $K$ w.r.t the larger graph $L$ (cf. \Cref{prop:persistent graph is a graph} and \Cref{thm:effective-persistent-preserve}). The connection with network circuit theory leads to our definition of a ``persistent'' Cheeger constant as well as to a novel persistent Cheeger-like inequality for a pair of graphs $K\hookrightarrow L$ (cf. \Cref{sec:cheeger}). 
    \end{enumerate}
     
    \item Finally, in \Cref{sec:filtration}, we consider $q$-th persistent Laplacians for filtrations of simplicial complexes (connected by inclusion morphisms). 
    We first describe an efficient algorithm to \emph{iteratively} compute the persistent Laplacian for all all pairs of complexes in a filtration.  We then discuss certain spectral stability results for the persistent Laplacian for filtrations of simplicial complexes.
\end{itemize}
Some technical details are relegated to the appendix. 

\section{The persistent Laplacian for simplicial pairs $K\hookrightarrow L$}  \label{sec:pairs}

In this section, after introducing some basic notions/definitions in \Cref{sec:basics}, we formulate the persistent Laplacian for simplicial pairs in \Cref{sec:definition of persistent laplacian} and present some basic properties of persistent Laplacians in \Cref{sec:basic properties of persistent Laplacian}.

\subsection{Basics}\label{sec:basics}
\paragraph{Simplicial complexes}An \emph{(abstract) simplicial complex} $K$ over a finite \emph{ordered} set $V$ is a collection of finite subsets of $V$ such that for any $\sigma\in K$, if $\tau\subseteq\sigma$, then $\tau\in K$.
Denote by $\mathbb N$ the set of non-negative integers. 
For each $q\in\mathbb{N}$,
an element $\sigma\in K$ is called a \emph{$q$-simplex} if $|\sigma|=q+1$, where we use $|A|$ to denote the cardinality of a set $A$. A $0$-simplex, usually denoted by $v$, is also called a \emph{vertex}. Denote by $S_q^K$ the set of $q$-simplices of $K$. Note that $S_0^K\subseteq V$. The \emph{dimension} of $K$, denoted by $\dim(K)$, is the largest $q$ such that $S_q^K\neq\emptyset$. A $1$-dim simplicial complex is also called a \emph{graph} and we often use $K=(V^K,E^K)$ to represent a graph, where $V^K\coloneqq S_0^K$ denotes the vertex set and $E^K\coloneqq S_1^K$ denotes the edge set.

An \emph{oriented} simplex, denoted by $[\sigma]$,
is a simplex $\sigma\in K$ with an ordering on its vertices. For simplicity of our presentation, we always assume that the ordering is inherited from the ordering of $V$. Let $\Bar{S}_q^K\coloneqq\left\{[\sigma]:\,\sigma\in S_q^K\right\}$. The $q$-th chain group $C_q^K\coloneqq C_q(K,\mathbb{R})$ of $K$ is the vector space over $\R$ with basis $\Bar{S}_q^K$. Let $n_q^K\coloneqq \dim C_q^K=\left|S_q^K\right|$. We define the boundary operator $\partial_q^K:C_q^K\rightarrow C_{q-1}^K$ by
\begin{equation}\label{eq:boundary map}
    \partial_q^K([v_0,\ldots,v_q])\coloneqq\sum_{i=0}^q(-1)^i[v_0,\ldots,\hat{v}_i,\ldots,v_q]
\end{equation}
for each $\sigma=[v_0,\ldots,v_q]\in \Bar{S}_q^K$, where $\hat{v}_i$ denotes the omission of the $i$-th vertex. The $q$-th homology group of $K$ is $H_q(K)=\frac{\ker\lc\partial_q^K\rc}{\mathrm{im}\lc\partial_{q+1}^K\rc}$ 
and $\beta_q^K\coloneqq\rank\lc H_q(K)\rc$ is its $q$-th Betti number.

A \emph{weight function} on a simplicial complex $K$ is any positive function $w^K:K\rightarrow(0,\infty)$. Throughout the paper, \textbf{each simplicial complex $K$ is (implicitly) endowed with a weight function $w^K$.} We call $K$ \emph{unweighted} if $w^K\equiv 1$.

\paragraph{Combinatorial Laplacian} Let $K$ be a simplicial complex with a weight function $w^K$. Given any $q\in\mathbb{N}$, let $w_q^K\coloneqq w^K|_{S_q^K}$ and define the inner product $\langle\cdot,\cdot\rangle_{w_q^K} $ on $C_q^K$ as follows:
\begin{equation}\label{eq:weight inner product}
    \langle [\sigma],[\sigma']\rangle_{w_q^K}\coloneqq\delta_{\sigma\sigma'}\cdot \lc w_q^K(\sigma)\rc^{-1},\,\,\forall\sigma,\sigma'\in S_q^K.
\end{equation} 

\begin{remark}
Consider the dual space of $C_q^K$: the cochain space $C^q(K)\coloneqq\mathrm{Hom}\lc C_q(K),\R\rc$. Then, $\langle\cdot,\cdot\rangle_{w_q^K}$ on $C_q(K)$ induces an inner product $\llangle\cdot,\cdot\rrangle_{{w}_q^K}$ on $C^q(K)$ such that 
\[\llangle f,g\rrangle_{{w}_q^K}=\sum_{\sigma\in S_q^K}w_q^K(\sigma)f([\sigma])g([\sigma]),\,\,\,\forall f,g\in C^q(K).\] 
This inner product $\llangle\cdot,\cdot\rrangle_{{w}_q^K}$ on $C^q(K)$ coincides with the one defined in \cite{horak2013spectra}, which explains the reciprocal in the definition \Cref{eq:weight inner product} of the inner product $\langle\cdot,\cdot\rangle_{w_q^K}$ on $C_q(K)$.
\end{remark}

We denote by $\lc\partial_q^K\rc^*:C_{q-1}^K\rightarrow C_{q}^K$ the adjoint of $\partial_q^K$ under these inner products. Then, we define the $q$-th \emph{(combinatorial) Laplacian}  $\Delta_q^K:C_q^K\rightarrow C_q^K$ as follows:
 \begin{equation}\label{eq:combinatorial Laplacian}
     \Delta_q^{K}\coloneqq \underbrace{\partial_{q+1}^{K}\circ \left(\partial_{q+1}^{K}\right)^\ast}_{\Delta^{K}_{q,\mathrm{up}}} + \underbrace{\left(\partial_q^K\right)^\ast\circ\partial_q^K}_{\Delta_{q,\mathrm{down}}^{K}},
 \end{equation}
where for convenience we have also defined the corresponding ``up'' and ``down'' Laplacians. By convention we let $\partial_0^K\coloneqq 0$ and thus $\Delta_0^{K}= \partial_{1}^{K}\circ \left(\partial_{1}^{K}\right)^\ast.$ When $K$ is a graph and $w_0^K\equiv 1$, $\Delta_0^{K}$ reduces to the graph Laplacian of the weighted graph $(K,w_1^K)$ \cite{chung1997sgt}. 

\begin{theorem}[\cite{eckmann1944harmonische}]\label{thm:ker lap = homology}
For each $q\in\mathbb N$, $\beta_q^K=\mathrm{nullity}\lc\Delta_q^K\rc$.
\end{theorem}

\paragraph{Simplicial pairs and simplicial filtrations} A \emph{simplicial pair}, denoted $K\hookrightarrow L$, consists of any pair $K$ and $L$ of simplicial complexes over the \emph{same finite ordered set} $V$ such that $K\subseteq L$, i.e., {$S_q^K\subseteq S_q^L$ for all $q\in\mathbb{N}$,} and $w^K=w^L|_K$. A \emph{simplicial filtration} $\mathbf{K} = \{K_t\}_{t\in T}$ is a set of simplicial complexes over the same finite ordered set $V$ indexed by a subset $T\subseteq\R$ such that for all $s\leq t\in T$, $K_s\hookrightarrow K_t$ is a simplicial pair. For an integer $q\geq 0$ and for any $s\leq t\in T$, via functoriality of Homology \cite{hatcher2000algebraic} one obtains a map $f^{s,t}_q: H_q(K_s)\rightarrow H_q(K_t)$ and the $q$-th \emph{persistent homology groups} are defined as the images of these maps. The $q$-th \emph{persistent Betti numbers} $\beta_q^{s,t}$ of $\mathbf{K}$ are in turn defined as the ranks of these groups. Of course when one is just presented with a simplicial pair $K\hookrightarrow L$, for each $q$ one also obtains the analogously defined $q$-th persistent Betti number $\beta^{K,L}_q$.

\subsection{Definition of the persistent Laplacian}\label{sec:definition of persistent laplacian}
Suppose that we have a simplicial pair $K\hookrightarrow L$ and that $q\in \mathbb{N}$. Consider the subspace 
\[C_q^{L,K}\coloneqq\left\{c \in C^L_q\,:\,\partial_q^L(c) \in C_{q-1}^K\right\}\subseteq C_q^L\] consisting of those $q$-chains in $C_{q}^L$ such that their images under the boundary operator $\partial_{q}^L$ is in the subspace $C_{q-1}^K$ of $C_{q-1}^L$. Let $n_q^{L,K}\coloneqq\dim\lc C_q^{L,K}\rc$.

Now, for each $q$ let $\partial_q^{L,K}$ denote the restriction of $\partial_q^L$ to $C_q^{L,K}$ so that we obtain the ``diagonal'' operators $\partial_q^{L,K}:C_q^{L,K}\rightarrow C_{q-1}^K$. As we mentioned earlier, for each $q$ both $C^K_q$ and $C^L_q$ are endowed with inner products $\langle\cdot,\cdot\rangle_{w_q^K}$ and $\langle\cdot,\cdot\rangle_{w_q^L}$ so that we can consider the adjoints of  $\partial_{q+1}^{L,K}$ and $\partial_q^L$. See the diagram below for the construction where the blue arrows signal the important part of the diagram: 

\begin{tikzcd}
\centering
     C_{q+1}^K\arrow{rr}{\partial_{q+1}^K}\arrow[hookrightarrow,dashed,gray]{dd} && C_q^K\arrow[rr,shift left=.75ex,blue,"\partial_q^K"]\arrow[hookrightarrow,dashed,gray]{dd}\arrow[dl,shift left=.75ex,blue,"\left(\partial_{q+1}^{L,K}\right)^*"] && C_{q-1}^K\arrow[hookrightarrow,dashed,gray]{dd}\arrow[ll,shift left=.75ex,blue,"\left(\partial_q^K\right)^*"]\\
      &C_{q+1}^{L,K}\arrow[hookrightarrow,dashed,gray]{dl}\arrow[ur,shift left=.75ex,blue,"\partial_{q+1}^{L,K}"]&& \,\,\,\,\,\,\,\,\,\, & \\
        C_{q+1}^L\arrow{rr}{\partial_{q+1}^L}  && C_q^L\arrow{rr}{\partial_q^L} && C_{q-1}^L
\end{tikzcd}

\noindent One can then define the $q$-th \emph{persistent Laplacian} \cite{wang2020persistent} $\Delta_q^{K,L}:C_q^K\rightarrow C_q^K$ by:
 \begin{equation}\label{eq:persist-Laplacian-no-basis}
     \Delta_q^{K,L}\coloneqq \underbrace{\partial_{q+1}^{L,K}\circ \left(\partial_{q+1}^{L,K}\right)^\ast}_{\Delta^{K,L}_{q,\mathrm{up}}} + {\left(\partial_q^K\right)^\ast\circ\partial_q^K},
 \end{equation}
where we have also defined the $q$-th \emph{up} persistent Laplacian $\Delta^{K,L}_{q,\mathrm{up}}$ with the same domain/codomain as $\Delta_q^{K,L}$. When $q=0$, since $\partial_0^K= 0$, $\Delta_0^{K,L}= \partial_{1}^{K,L}\circ \left(\partial_{1}^{K,L}\right)^\ast=\Delta^{K,L}_{0,\mathrm{up}}.$

\begin{example}[Trivial cases]\label{ex:trivial cases}
\begin{enumerate}
    \item When $C_{q+1}^{L,K}=\{0\}$, $\partial_{q+1}^{L,K}=0$ and thus $\Delta^{K,L}_{q,\mathrm{up}}=0$.
    \item When $K=L$, then obviously $\Delta_q^{K,L}=\Delta_q^L$, the usual Laplacian on $L$. 
    \item If $S_q^K=S_q^L$, then $\Delta_{q,\mathrm{up}}^{K,L}=\Delta_{q,\mathrm{up}}^L$. In particular, if $S_0^K=S_0^L$, then $\Delta_0^{K,L}=\Delta_0^L$. If furthermore $S_{q-1}^K=S_{q-1}^L$, then $\Delta_{q,\mathrm{down}}^{K}=\Delta_{q,\mathrm{down}}^L$ and thus $\Delta_q^{K,L}=\Delta_q^L$. 
\end{enumerate}
\end{example}

Obviously, $\Delta_q^{K,L}$ is a self-adjoint, non-negative and compact operator on $C_q^K$ and thus has non-negative real eigenvalues. We denote by
$0\leq \lambda_{q,1}^{K,L}\leq \lambda_{q,2}^{K,L}\leq\ldots\leq \lambda_{q,n_q^K}^{K,L}$
the eigenvalues of $\Delta_q^{K,L}$ sorted in increasing order, including repetitions.

\subsection{Basic properties of the persistent Laplacian}\label{sec:basic properties of persistent Laplacian}
We now show some basic properties of $\Delta_q^{K,L}$. All proofs are given in \Cref{sec:proofs from the paper}.

\begin{lemma}\label{lm:connected component}
Suppose $L$ has $n$ connected components $L_1,\ldots,L_n$. Suppose $K$ only intersects the first $m$ connected components. Let $K_i\coloneqq K\cap L_i$ for each $i=1,\ldots,m$. Then, $\Delta_q^{K,L}$ is the direct sum of persistent Laplacians $\Delta_q^{K_i,L_i}$ on $C_q^{K_i}$ for $i=1,\ldots,m$, i.e., $\Delta_q^{K,L}=\bigoplus_{i=1}^m\Delta_q^{K_i,L_i}.$
\end{lemma}

Given a graph $K$, the multiplicity of the 0 eigenvalue of $\Delta_0^K$ coincides with the number of connected components of $K$ \cite{marsden2013eigenvalues}. The following result is a persistent version of this.

\begin{theorem}\label{thm:connected component}
The eigenvalues of $\Delta_0^{K,L}$ satisfy the following basic properties.
\begin{enumerate}
    \item $\lambda_{0,1}^{K,L}=0$; and if $L$ is connected, then $\lambda_{0,2}^{K,L}>0$.
    \item Let $m$ be the multiplicity of the $0$ eigenvalue of $\Delta_0^{K,L}$, then $K$ intersects exactly $m$ connected components of $L$.
\end{enumerate}
\end{theorem}

We have a complete description of the behavior of the up persistent Laplacian on \emph{interior simplices}, where a $q$-simplex $\sigma\in S_q^K$ is called an {interior simplex} if $\sigma$ only shares cofaces with $q$-simplices in $K$, i.e., $\forall \sigma'\in S_q^L,$ if $\sigma\cup \sigma'\in S_{q+1}^L,$ then $\sigma'\in S_q^K.$

\begin{theorem}\label{thm:pers_interior}
Let $c^L\in C_q^L$ and let $c^K$ be the image of $c^L$ under the orthogonal projection $C_q^L\rightarrow C_q^K$. Then, for any interior simplex $\sigma\in S_q^K$, we have that
\[\left\langle \Delta_{q,\mathrm{up}}^Lc^L, [\sigma] \right\rangle_{w_{q}^L}=\left\langle \Delta_{q,\mathrm{up}}^{K,L}c^K, [\sigma] \right\rangle_{w_{q}^K}. \]
\end{theorem}

The following result showing persistent Laplacians recover persistent Betti numbers was mentioned in passing and without proof in \cite{wang2020persistent} and was also implicitly contained in pages 10 and 11 of \cite{lapslides}. 
We give a full proof in \Cref{sec:proofs from the paper}.

\begin{theorem}\label{thm:pers-betti-pers-lap} 
For each integer $q\geq 0$, we have that $\beta_q^{K,L}=\mathrm{nullity}\lc\Delta_q^{K,L}\rc$.
\end{theorem}

\section{A first algorithm for computing a matrix representation of $\Delta_q^{K,L}$}\label{sec:first algo persistent Laplacian}

In this section, we first provide a matrix representation $\lap_q^{K,L}$ of $\Delta_q^{K,L}$ given the canonical basis $\bar{S}_q^K$ of $C_q^K$ and then devise an algorithm for computing $\lap_q^{K,L}$. \footnote{In \cite{wang2020persistent} it is suggested that the $q$-th persistent Laplacian $\Delta_q^{K,L}$ can be computed by (i) taking a certain submatrix of the boundary operator and then (ii) multiplying it by its transpose. However, simply following these two steps does not yield a correct algorithm. The calculation of the matrix form of the persistent Laplacian turned out to be rather subtle as shown in \Cref{thm:weighted-basis-persistent-Laplcacian}; see also \Cref{sec:wrong algorithm example} for details. } 

\medskip
\noindent\textbf{Note:} For simplicity, given a simplicial pair $K\hookrightarrow L$, for each $q\in\mathbb N$ we assume an ordering $\bar{S}_q^L=\{[\sigma_i]\}_{i=1}^{n_q^L}$ on $\Bar{S}_q^L$  such that $\bar{S}_q^K=\{[\sigma_i]\}_{i=1}^{n_q^K}$. Unless otherwise specified, matrix representations of operators between chain groups are \emph{always} from such orderings on canonical bases $\Bar{S}_q^K$ and $\Bar{S}_q^L$ of $C_q^K$ and $C_q^L$, respectively.

\begin{theorem}\label{thm:weighted-basis-persistent-Laplcacian}
Assume that $n_{q+1}^{L,K}\coloneqq\dim\lc C_{q+1}^{L,K}\rc>0$. Choose any basis of $C_{q+1}^{L,K}\subseteq C_{q+1}^{L}$ represented by a column matrix $Z\in\mathbb{R}^{n_{q+1}^L\times n_{q+1}^{L,K}}$. Let $B_q^K$ and $B_{q+1}^{L,K}$ be matrix representations of boundary maps $\partial_q^K$ and $\partial_{q+1}^{L,K}$, respectively. Let $W_q^K$ (or $W_q^L$) denote the diagonal weight matrix representation of $w_q^K$ (or $w_q^L$). Then, the matrix representation $\lap_q^{K,L}$ of $\Delta_q^{K,L}$ is expressed as follows: 
\begin{equation} \label{eq:weighted-perLoperator}
\lap_q^{K,L}=  \underbrace{B_{q+1}^{L,K}\left(Z^\mathrm{T}\lc W_{q+1}^L\rc^{-1}Z\right)^{-1}\left( B_{q+1}^{L,K}\right)^\mathrm{T}\lc W_q^K\rc^{-1}}_{\lap_{q,\mathrm{up}}^{K,L}}+\underbrace{W_q^K\left( B_q^K\right)^\mathrm{T} \lc W_{q-1}^K\rc^{-1}B_q^K}_{\lap_{q,\mathrm{down}}^{K}}.
\end{equation}
Moreover, $\lap_q^{K,L}$ is invariant under the choice of basis for $C_{q+1}^{L,K}$.
\end{theorem}

\begin{remark}[Matrix representations of combinatorial Laplacians]\label{rmk:w=1 weighted laplacian}
When $K=L$, \Cref{eq:weighted-perLoperator} reduces to the matrix representation of the combinatorial Laplacian:
\[\lap_q^{K}\coloneqq \underbrace{B_{q+1}^{K} W_{q+1}^K\left( B_{q+1}^{K}\right)^\mathrm{T}\lc W_q^K\rc^{-1}}_{\lap^{K}_{q,\mathrm{up}}} + \underbrace{W_q^K\left( B_q^K\right)^\mathrm{T} \lc W_{q-1}^K\rc^{-1}B_q^K}_{\lap_{q,\mathrm{down}}^{K}}.\]
Since $B_{q+1}^{K} W_{q+1}^K\left( B_{q+1}^{K}\right)^\mathrm{T}\lc W_q^K\rc^{-1}=\lc W_q^K\rc^{\frac{1}{2}}\lc \lc W_q^K\rc^{-\frac{1}{2}} B_{q+1}^{K} W_{q+1}^K\left( B_{q+1}^{K}\right)^\mathrm{T} \lc W_q^K\rc^{-\frac{1}{2}}\rc \lc W_q^K\rc^{-\frac{1}{2}},$ $\lap_{q,\mathrm{up}}^{K}$ is of the form $W^{-1}PW$ where $P$ is symmetric positive semi-definite and $W$ is a positive diagonal matrix. The same result holds for down Laplacians, up persistent Laplacians, and (persistent) Laplacians.
Note that if $w_q^K\equiv 1$, then $\lap_q^{K}={B_{q+1}^{K}W_{q+1}^K \left(B_{q+1}^{K}\right)^\T }+ {\left( B_q^K\right)^\mathrm{T} \lc W_{q-1}^K\rc^{-1}B_q^K}$ is itself a symmetric positive semi-definite matrix.
\end{remark}

To prove the theorem, we need the following result:
\begin{lemma}\label{lm:linear map inner product basis}
Let $f:(\mathbb{R}^n,W_n)\rightarrow(\mathbb{R}^m,W_m)$ be a linear map where $W_n\in \mathbb{R}^{n\times n}$ and $W_m\in \R^{m\times m}$ denote the inner product matrices. Let $F\in \mathbb{R}^{m\times n}$ denote the matrix representation of $f$. Then, the matrix representation $F^*$ of the adjoint $f^*$ of $f$ is $W_n^{-1}F^\T W_m $.
\end{lemma}
\begin{proof}
For any $x=(x_1,\ldots,x_n)^\T\in \mathbb{R}^{n}$ and $y=(y_1,\ldots,y_m)^\T\in \mathbb{R}^{m}$, we have that
\[\left\langle fx,y\right\rangle_{\R^m}=\lc Fx\rc^\mathrm{T}W_my=x^\mathrm{T}F^\mathrm{T}W_my,\text{ and }\,\left\langle x,f^*y\right\rangle_{ \R^n}=x^\mathrm{T}W_nF^*y.\]
Since $\left\langle fx,y\right\rangle_{\R^m}=\left\langle x,f^*y\right\rangle_{ \R^n}$ and $x,y$ are arbitrary, we must have that 
$F^*=W_n^{-1}F^\T W_m $.
\end{proof}

\begin{proof}[Proof of \Cref{thm:weighted-basis-persistent-Laplcacian}]
Base on our choice of bases for $C_{q+1}^{L,K}$, $C_q^K$ and $C_{q-1}^K$, the corresponding inner product matrices are $Z^\mathrm{T}\lc W_{q+1}^L\rc^{-1}Z$, $\lc W_{q}^K\rc^{-1}$ and $\lc W_{q-1}^K\rc^{-1}$, respectively. By \Cref{lm:linear map inner product basis}, the matrix representation for $(\partial_{q+1}^{L,K})^\ast$ is $\left(Z^\mathrm{T}\lc W_{q+1}^L\rc^{-1}Z\right)^{-1}( B_{q+1}^{L,K})^\mathrm{T}\lc W_{q}^K\rc^{-1}$ and the matrix representation for $(\partial_{q}^{K})^\ast$ is $W_{q}^K\lc B_q^K\rc^\T \lc W_{q-1}^K\rc^{-1}$. By \Cref{eq:persist-Laplacian-no-basis}, we have
\[\lap_q^{K,L}=  B_{q+1}^{L,K}\left(Z^\mathrm{T}\lc W_{q+1}^L\rc^{-1}Z\right)^{-1}\left( B_{q+1}^{L,K}\right)^\mathrm{T}\lc W_q^K\rc^{-1}+W_q^K\left( B_q^K\right)^\mathrm{T} \lc W_{q-1}^K\rc^{-1}B_q^K.\]

Since $\partial_{q+1}^{L,K}\left(\partial_{q+1}^{L,K}\right)^\ast$ is a self-operator on $C_q^K$, its matrix representation $\lap_q^{K,L}$ only depends on the choice of basis of $C_q^K$ and it is thus independent of the choice of basis of $ C_{q+1}^{L,K}$. 
\end{proof}

\paragraph*{An algorithm for computing the matrix representation of $\Delta_q^{K,L}$} We use the symbol $[n]$ to denote the set $\{1,\ldots,n\}$ for a positive integer $n$. 
We first introduce a notation for representing submatrices. Let $M\in\R^{m\times n}$ be a real matrix and let $\emptyset\neq I\subseteq [m]$ and $\emptyset\neq J\subseteq[n]$. We denote by $M(I,J)$ the submatrix of $M$ consisting of those rows and columns indexed by $I$ and $J$, respectively. Moreover, we use $M(:,J)$ (or $M(I,:)$) to denote $M([m],J)$ (or $M(I,[n])$).

By \Cref{thm:weighted-basis-persistent-Laplcacian}, to compute a matrix representation of $\Delta_q^{K,L}$, the key is to produce a basis {(i.e., $Z$)} for $C_{q+1}^{L,K}$. Let $B_{q+1}^L\in\R^{n_q^L\times n_{q+1}^L}$ be the matrix representation of the boundary map $\partial_{q+1}^L$. 
We assume that $n_q^K<n_q^L$ since the case $n_q^K=n_q^L$ is trivial (cf. \Cref{ex:trivial cases}). Then, the
following lemma (proof in \Cref{sec:proofs from the paper}) suggests a way of {constructing $Z$} from $B_{q+1}^L$.

\begin{lemma}\label{lm:image in subspace}
Let $D_{q+1}^L\coloneqq B_{q+1}^L\lc [n_q^L]\backslash[n_q^K],: \rc$. Then, there exists a non-singular matrix $Y\in\R^{n_{q+1}^L\times n_{q+1}^L}$ such that $R_{q+1}^L\coloneqq D_{q+1}^LY$ is column reduced\footnote{We say a matrix is \emph{column reduced}, if for each two non-zero columns, their indices of lowest non-zero elements are different.}. Moreover, let $I\subseteq[n_{q+1}^L]$ be the index set of $0$ columns of $R_{q+1}^L$. The following hold:
\begin{enumerate}
    \item If $I=\emptyset$, then $C_{q+1}^{L,K}=\{0\}$;
    \item If $I\neq\emptyset$, let $Z\coloneqq Y(:,I)$, then columns of $Z$ constitute a basis of $C_{q+1}^{L,K}$. 
\end{enumerate}
Moreover, if $I\neq\emptyset$, then $B_{q+1}^{L,K}\coloneqq\lc B_{q+1}^LY\rc\lc[n_q^K],I\rc$ is the matrix representation of $\partial_{q+1}^{L,K}$.
\end{lemma}

We can apply a column reduction process (e.g., Gaussian elimination) to $D_{q+1}^L$ to obtain $Y\in\R^{n_{q+1}^L\times n_{q+1}^L}$ and $R_{q+1}^L\coloneqq D_{q+1}^LY$ requested in \Cref{lm:image in subspace}. See \Cref{algo-pers-lap} for a pseudocode for computing $\lap_q^{K,L}$ based on \Cref{lm:image in subspace}.

\begin{algorithm}[htb]
 \caption{Persistent Laplacian: matrix representation}
 \label{algo-pers-lap}
 \begin{algorithmic}[1]
 \STATE \textbf{Data:} $B_q^K,B_{q+1}^L,W_{q-1}^K,W_q^K$ and $W_{q+1}^L$
 \STATE \textbf{Result:} $\lap_q^{K,L}$
 \STATE compute $\lap_{q,\mathrm{down}}^K$ from $B_q^K,W_{q-1}^K$ and $W_q^K$
 \IF{$n_q^K==n_q^L$}
 \STATE compute $\lap_{q,\mathrm{up}}^L$ from $B_{q+1}^L,W_{q}^K$ and $W_{q+1}^L$;  ~~\RETURN $ \lap_{q,\mathrm{up}}^L+\lap_{q,\mathrm{down}}^K$
 \ENDIF
 \STATE $D_{q+1}^L=B_{q+1}^L\left([n_q^L]\backslash[n_q^K],:\right) $
 \STATE $(R_{q+1}^L,Y)=\mathrm{ColumnReduction}(D_{q+1}^L) $
 \STATE $I\gets$ index set corresponding to the all-zero columns of $R_{q+1}^L$
 \IF{$I==\emptyset$} 
 \RETURN $\lap_{q,\mathrm{down}}^K$
 \ENDIF
 \STATE $Z = Y(:,I)$
 \STATE $B_{q+1}^{L,K}=\left(B_{q+1}^LY\right)\lc [n_q^K],I\rc$
 \STATE \RETURN $ B_{q+1}^{L,K}\left(Z^\mathrm{T}\lc W_{q+1}^L\rc^{-1}Z\right)^{-1}\left( B_{q+1}^{L,K}\right)^\mathrm{T}\lc W_q^K\rc^{-1}+\lap_{q,\mathrm{down}}^K$
\end{algorithmic}
\end{algorithm}

\paragraph{Complexity analysis} The computation of $\lap_{q,\mathrm{down}}^K$ takes time $O\left(\lc n_{q}^K\rc^2\right)$ (See \Cref{sec:computation of up and down} for details). The size of $D_{q+1}^L$ is $\lc n_q^L-n_q^K\rc\times n_{q+1}^L$; thus the column reduction process takes time $O\left((n_q^L-n_q^K)\lc n_{q+1}^L\rc^2\right)$. Computing the product $B_{q+1}^LY$ takes time $O\left(n_q^L\lc n_{q+1}^L\rc^2\right)$. The size of $Z$ is $n_{q+1}^L\times |I|$, where $|I|\leq n_{q+1}^L$. Then, computing $\left(Z^\mathrm{T}\lc W_{q+1}^L\rc^{-1}Z\right)^{-1}$ takes time $O\left(\lc n_{q+1}^L\rc^3\right)$. The product $B_{q+1}^{L,K}\left(Z^\mathrm{T}\lc W_{q+1}^L\rc^{-1}Z\right)^{-1}\left( B_{q+1}^{L,K}\right)^\mathrm{T}\lc W_q^K\rc^{-1}$ can be computed in time
$O\left(n_q^K\lc n_{q+1}^L\rc^2\right)$. 
Hence \Cref{algo-pers-lap} takes 
$O\left(n_q^L\lc n_{q+1}^L\rc^2+\lc n_{q+1}^L\rc^3+\lc n_{q}^K\rc^2\right)$ total time. 
One can also improve this time complexity by using fast matrix-multiplication to both perform reduction and compute multiplication/inverse.  We omit the details.

\section{Schur complement, persistent Laplacian and  implications}\label{sec:schur complement}

Let $M\in\R^{n\times n}$ be a block matrix $M=\begin{pmatrix}
A &  B\\
C & D
\end{pmatrix}$ where $D\in\R^{d\times d}$ is a square matrix. Then, \emph{the (generalized) Schur complement of $D$ in $M$} \cite{carlson1974generalization}, denoted by $M/D$, is $M/D\coloneqq A-BD^\dagger C$, 
where $D^\dagger$ is the Moore-Penrose generalized inverse of $D$. Note that having $D$ to be the bottom right submatrix is only for notational simplicity. Schur complement is defined for any principal submatrix. More precisely, let $\emptyset\neq I\subsetneq [n]$ be a proper subset. Then, the (generalized) Schur complement of $M(I,I)$ in $M$ is defined as
\begin{equation}\label{eq:Schur complement index formula}
    M/M(I,I)\coloneqq M([n]\backslash I,[n]\backslash I)-M([n]\backslash I,I)M(I,I)^\dagger M(I,[n]\backslash I).
\end{equation}

Now we introduce some useful properties of the Schur complement.

\begin{definition}[Proper submatrices]
Let  $M=\begin{pmatrix}
A &  B\\
C & D
\end{pmatrix}$ be a square block matrix where both $A$ and $D$ are square matrices. The submatrix $D$ is \emph{proper} in $M$ if $\ker(D)\subseteq\ker(B)$ and $\ker\lc D^\T\rc\subseteq\ker\lc C^\T\rc$. 
\end{definition}

\begin{lemma}[Positive semi-definite matrices]\label{lm:psd-proper}
Let $P$ be a positive semi-definite block matrix
$P =\begin{pmatrix}
A &  B\\
C & D
\end{pmatrix}$ such that $A$ and $D$ are square matrices. Let $W$ be a positive diagonal matrix and we write $W$ as a block matrix
$W =\begin{pmatrix}
W_1 &  0\\
0 & W_2
\end{pmatrix}$ such that $W_1$ and $W_2$ have the same sizes as $A$ and $D$, respectively. Consider $M\coloneqq W^{-1}PW=\begin{pmatrix}
W_1^{-1}AW_1 &  W_1^{-1}BW_2\\
W_2^{-1}CW_1 & W_2^{-1}DW_2
\end{pmatrix}$. Then, $W_2^{-1}DW_2$ is proper in $M$ and $M/(W_2^{-1}DW_2)=W_1^{-1}(P/D)W_1.$
\end{lemma}

\begin{lemma}[{\cite[Theorem 1]{carlson1974generalization}}]\label{lm:schur-complement-rank}
Let $M$ be a square block matrix
$M =\begin{pmatrix}
A &  B\\
C & D
\end{pmatrix}$ such that $A$ and $D$ are square matrices. Then, $\mathrm{rank}(M) \geq \mathrm{rank}(D)+\mathrm{rank}(M/D).$
\end{lemma}

\begin{lemma}[Quotient Formula {\cite[Theorem 4]{carlson1974generalization}}]\label{lm:quotient formula}
Let $M,D$ and $H$ be square matrices with the following block structures:  
$M =\begin{pmatrix}
A &  B\\
C & D
\end{pmatrix}\text{ and }D =\begin{pmatrix}
E &  F\\
G & H
\end{pmatrix}.$
If $D$ is proper in $M$ and $H$ is proper in $D$, then $D/H$ is proper in $M/H$ and 
$M/D = (M/H)/(D/H).$
\end{lemma}

\begin{lemma}[Eigenvalue interlacing property]\label{lm:eigen-interlacing}
Let $M=W^{-1}PW$ be as in \Cref{lm:psd-proper}. Suppose that the size of $M$ is $n\times n$ and the size of $D$ is $d\times d$. Then, 
\begin{equation}\label{eq:interlacing}
    \lambda_k(M)\leq \lambda_k(M/(W_2^{-1}DW_2))\leq \lambda_k(W_1^{-1}AW_1),\quad \forall 1\leq k\leq n-d,
\end{equation}
where $\lambda_k(A)$ denotes the $k$-th smallest eigenvalue of $A$ (counted with multiplicity). 
\end{lemma}
See \Cref{sec:proofs from the paper} for proofs of \Cref{lm:psd-proper} and \Cref{lm:eigen-interlacing}.

\subsection{Up-persistent Laplacian as a Schur complement}\label{sec:persistent laplacian as schur complement}
For a simplicial pair $K\hookrightarrow L$, recall from \Cref{sec:first algo persistent Laplacian} that for each $q\in\mathbb N$ we assume an ordering $\bar{S}_q^L=\{[\sigma_i]\}_{i=1}^{n_q^L}$ on $\Bar{S}_q^L$  such that $\bar{S}_q^K=\{[\sigma_i]\}_{i=1}^{n_q^K}$. Given such orderings on canonical bases of $C_q^K$ and $C_q^L$, the matrix representation $\lap_{q,\mathrm{up}}^{K,L}$ of $\Delta_{q,\mathrm{up}}^{K,L}:C_q^K\rightarrow C_q^K$ is related to the matrix representation $\lap_{q,\mathrm{up}}^L$ of $\Delta_{q,\mathrm{up}}^L:C_q^L\rightarrow C_q^L$ via the Schur complement as follows:

\begin{theorem}[Up-persistent Laplacian as Schur complement]\label{thm:persis-Laplacian-schur-formula}
Let $K\hookrightarrow L$ be a simplicial pair. Assume that $n_q^K<n_q^L$ and let $I_K^L\coloneqq [n_q^L]\backslash[n_q^K]$. Then,
\begin{equation}\label{eq:schur complement up laplacian}
    \lap_{q,\mathrm{up}}^{K,L}=\lap_{q,\mathrm{up}}^L/\lap_{q,\mathrm{up}}^L\lc I_K^L,I_K^L\rc.
\end{equation}
\end{theorem}

To prove the above theorem, we first need the following lemma (whose proof is given in \Cref{sec:proofs from the paper}) which relates Schur complements with a certain matrix operation.

\begin{lemma}\label{lm:schur-complement-matrix-operation}
Let $B\in\R^{n\times m}$ be a block matrix $B=\begin{pmatrix}
B_1\\
B_2
\end{pmatrix}$, where $B_1\in\R^{d\times m}$ for some $1\leq d<n$. Let $W_1\in \R^{d\times d}$ and $W_2\in\R^{(n-d)\times(n-d)}$ be non-singular diagonal matrices and let $W =\begin{pmatrix}
W_1 &  0\\
0 & W_2
\end{pmatrix}$. Let $M\coloneqq BB^\mathrm{T}W$, which is a block matrix 
\[M =\begin{pmatrix}
M_{11} &  M_{12}\\
M_{21} & M_{22} 
\end{pmatrix}=\begin{pmatrix}
B_1B_1^\T W_1 &  B_1B_2^\T W_2\\
B_2B_1^\T W_1 & B_2B_2^\T W_2 
\end{pmatrix}.\]
If $B_2$ has \emph{full column rank}, then $M/M_{22}=0$. Otherwise, for any non-singular block matrix $Y=\begin{pmatrix}
Y_1 & Y_2 \end{pmatrix}\in\R^{m\times m}$, if $B_2Y_1=0$ and $B_2Y_2$ has \emph{full column rank}, then
$M/M_{22}=B_1Y_1\lc Y_1^\mathrm{T}Y_1\rc^{-1}(B_1Y_1)^\mathrm{T}W_1.$
\end{lemma}

\begin{proof}[Proof of \Cref{thm:persis-Laplacian-schur-formula}]
Let $B\coloneqq B_{q+1}^L\lc W_{q+1}^L\rc^\frac{1}{2}$, $W\coloneqq W_q^L$ and $W_1\coloneqq W([n_q^K],[n_q^K])=W_q^K$. Set $B_1\coloneqq B\lc[n_q^K],:\rc$ and $B_2\coloneqq B\lc[n_q^L]\backslash[n_q^K],:\rc$. Then, $B=\begin{pmatrix}
B_1\\
B_2
\end{pmatrix}$. Note that $B_2=D_{q+1}^L\lc W_{q+1}^L\rc^\frac{1}{2}$ using notations in \Cref{lm:image in subspace}. By \Cref{lm:image in subspace}, there exists a non-singular matrix $\hat{Y}\in\R^{n_{q+1}^L\times n_{q+1}^L}$ such that $R_{q+1}^L\coloneqq D_{q+1}^L\hat{Y}$ is column reduced. Let $Y\coloneqq \lc W_{q+1}^L\rc^{-\frac{1}{2}}\hat{Y}$, which is still non-singular. Then, 
\[R_{q+1}^L= D_{q+1}^L\hat{Y}=D_{q+1}^L \lc W_{q+1}^L\rc^{\frac{1}{2}}\lc W_{q+1}^L\rc^{-\frac{1}{2}}\hat{Y}=B_2Y.\]  
Let $I\subseteq[n_{q+1}^L]$ be the index set of $0$ columns of $R_{q+1}^L$. If $I=\emptyset$, then by \Cref{lm:image in subspace} we have that $C_{q+1}^{L,K}=\{0\}$ and thus $\lap_{q,\mathrm{up}}^{K,L}=0$. On the other hand, $I=\emptyset$ implies that $B_2$ has full column rank. Let $M\coloneqq BB^\T W$. Then, we have that
\[M =B_{q+1}^L W_{q+1}^L\lc B_{q+1}^L\rc^\T W_q^L =\lap_{q,\mathrm{up}}^L\]
and thus $M_{22}=\lap_{q,\mathrm{up}}^L\lc  I_K^L,I_K^L\rc$. Then by \Cref{lm:schur-complement-matrix-operation}, we have that 
\[\lap_{q,\mathrm{up}}^L/\lap_{q,\mathrm{up}}^L\lc I_K^L,I_K^L\rc=M/M_{22}=0=\lap_{q,\mathrm{up}}^{K,L}.\]

Now, we assume that $I\neq\emptyset$. Without loss of generality, we assume that $I=[n_{q+1}^{L,K}]\subseteq[n_{q+1}^L]$ (otherwise we multiply $Y$ by a permutation matrix). Let $Y_1\coloneqq Y\lc:,I\rc=\lc W_{q+1}^L\rc^{-\frac{1}{2}}Z$ where $Z$ is a column matrix representing a basis of $C_{q+1}^{L,K}$ (cf. \Cref{lm:image in subspace}). Let $Y_2\coloneqq Y\lc:,[n_{q+1}^L]\backslash I\rc$. Then, $Y=\begin{pmatrix}Y_1& Y_2\end{pmatrix}$ is a block matrix such that $B_2Y_1=R_{q+1}^L(:,I)=0$ and that $B_2Y_2=R_{q+1}^L\lc:,[n_{q+1}^L]\backslash I\rc$ has full column rank. Then, by \Cref{lm:schur-complement-matrix-operation}, we have that 
\begin{align*}
    &\lap_{q,\mathrm{up}}^L/\lap_{q,\mathrm{up}}^L\lc I_K^L,I_K^L\rc=M/M_{22}=B_1Y_1\lc Y_1^\mathrm{T}Y_1\rc^{-1}(B_1Y_1)^\mathrm{T}W_1\\
    =&B_1\lc W_{q+1}^L\rc^{-\frac{1}{2}}Z\lc \lc \lc W_{q+1}^L\rc^{-\frac{1}{2}}Z\rc^\mathrm{T}\lc W_{q+1}^L\rc^{-\frac{1}{2}}Z\rc^{-1}\lc B_1\lc W_{q+1}^L\rc^{-\frac{1}{2}}Z\rc^\mathrm{T}W_q^K\\
    =&B_1\lc W_{q+1}^L\rc^{-\frac{1}{2}}Z\lc Z^\T\lc W_{q+1}^L\rc^{-1}Z\rc^{-1}\lc B_1\lc W_{q+1}^L\rc^{-\frac{1}{2}}Z\rc^\mathrm{T}W_q^K.
\end{align*}
Note also that $B_{q+1}^{L,K}=B_{q+1}^L([n_q^K],:)Z = B_1\lc W_{q+1}^L\rc^{-\frac{1}{2}}Z$. Then, by \Cref{lm:image in subspace} we have that 
\[\lap_{q,\mathrm{up}}^{K,L}=B_{q+1}^{L,K}\lc Z^\T\lc W_{q+1}^L\rc^{-1}Z\rc^{-1}\left( B_{q+1}^{L,K}\right)^\mathrm{T}W_q^K=\lap_{q,\mathrm{up}}^L/\lap_{q,\mathrm{up}}^L\lc I_K^L,I_K^L\rc.\]
This finishes the proof of \Cref{thm:persis-Laplacian-schur-formula}. 
\end{proof}

\subsection{Fast computation of the matrix representation of $\Delta_q^{K,L}$}\label{sec:fast computation of persistent laplacian}
For a simplicial pair $K\hookrightarrow L$, by \Cref{thm:persis-Laplacian-schur-formula}, we now simply
compute $\lap_{q,\mathrm{up}}^{K,L}$ via \Cref{eq:schur complement up laplacian} using only Schur complement computations, which then give us $\lap_{q}^{K,L}=\lap_{q,\mathrm{up}}^{K,L}+\lap_{q,\mathrm{down}}^{K}$.
A pseudocode for this simple algorithm is given in \Cref{algo-pers-lap-schur}.

\begin{algorithm}[H]
 \caption{Persistent Laplacian: matrix representation via Schur complement}
 \label{algo-pers-lap-schur}
 \begin{algorithmic}[1]
 \STATE \textbf{Data:} $B_q^K,B_{q+1}^L,W_{q-1}^K,W_q^K,W_q^L$ and $W_{q+1}^L$
 \STATE \textbf{Result:} $\lap_q^{K,L}$
  \STATE Compute $\lap_{q,\mathrm{down}}^K$ from $B_q^K,W_{q-1}^K$ and $W_q^K$
 \STATE Compute $\lap_{q,\mathrm{up}}^L$ from $B_{q+1}^L,W_{q}^L$ and $W_{q+1}^L$
 \IF{$n_q^K==n_q^L$}
 \RETURN $ \lap_{q,\mathrm{up}}^L+\lap_{q,\mathrm{down}}^K$
 \ENDIF
 \STATE{$\lap_{q,\mathrm{up}}^{K,L}=\lap_{q,\mathrm{up}}^L/\lap_{q,\mathrm{up}}^L\lc I_K^L,I_K^L\rc$}
 \STATE{\RETURN $\lap_{q,\mathrm{up}}^{K,L}+\lap_{q,\mathrm{down}}^K$}
\end{algorithmic}
\end{algorithm}

\paragraph{Time complexity} Computing $\lap_{q,\mathrm{up}}^L$ takes time $O\left(n_{q+1}^L\right)$ and computing $\lap_{q,\mathrm{down}}^K$ takes $O\left(\lc n_{q}^K\rc^2\right)$ (see \Cref{sec:computation of up and down} for details).  The Schur complement $\lap_{q,\mathrm{up}}^L/\lap_{q,\mathrm{up}}^L\lc I_K^L,I_K^L\rc$ takes time 
$$O\lc\lc n_q^K\rc ^2+\lc n_q^L-n_q^K\rc^3+n_q^K\lc n_q^L-n_q^K\rc^2\rc=O\lc\lc n_q^L\rc^3\rc$$
to compute. Hence the total time complexity of computing $\lap_{q}^{K,L}$ via \Cref{eq:schur complement up laplacian} is $O\left(\lc n_q^L\rc^3+n_{q+1}^L \right)$, which is more efficient than the complexity of \Cref{algo-pers-lap}, $O\left(n_q^L\lc n_{q+1}^L\rc^2+\lc n_{q+1}^L\rc^3+\lc n_{q}^K\rc^2\right)$, when $n_q^L= O(n_{q+1}^L)$. 
By using fast matrix multiplication algorithm (which takes $O(r^\omega)$, $\omega < 2.373$, to multiply two $r\times r$ matrices), this time complexity can be improved to $O\left(\lc n_q^L\rc^\omega+n_{q+1}^L\right)$. 

\paragraph{Computation of persistent Betti numbers}\label{betti number complexity} By \Cref{thm:pers-betti-pers-lap}, we can compute the persistent Betti number $\beta_q^{K,L}$ in the following manner: we first compute $\lap_q^{K,L}$ and then compute $\beta_q^{K,L}=\mathrm{nullity}\lc \lap_q^{K,L}\rc$. 
Since calculating the nullity of an $n_q^K\times n_q^K$ square matrix can be done in time $O\lc \lc n_q^K\rc^\omega\rc = O\lc \lc n_q^L\rc^\omega\rc$, we obtain a method for computing the persistent Betti number in time $O\left(\lc n_q^L\rc^\omega +n_{q+1}^L\right)$ (which is $O\left(\lc n_q^L\rc^\omega\right)$ if $n_q^L = O\lc n_{q+1}^L\rc$). 
Currently, the existing approach in the literature to compute the persistent Betti numbers is through computing the persistent homology of the pair $K\hookrightarrow L$ using boundary matrices $B_{q+1}^L$ and $B_q^K$, which can be done in $O\left(\lc n_{q}^L\rc^2n_{q+1}^L+\lc n_{q-1}^K\rc^2n_{q}^K \right)$ time or in $O\left(\lc n_{q}^L\rc^{\omega-1}n_{q+1}^L+\lc n_{q-1}^K\rc^{\omega-1}n_{q}^K\right)$ (if we assume that $n_q^L = O(n_{q+1}^{L})$ and $n_{q-1}^K=O\lc n_q^K\rc$) using earliest basis (via fast matrix multiplication) approach \cite{BCCDW12}. 
Our new algebraic formulation of persistent Laplacian (via Schur complement) thus also leads to a faster algorithm to compute the persistent Betti number for a pair of spaces {\bf{for the setting when {$n_q^L = O(n_{q+1}^{L})$}}}. Note that the condition $n_q^L = O(n_{q+1}^L)$ holds in many practical scenarios, especially for the popular Rips or \v{C}ech complexes and their variants.
Given that this new algorithm is fundamentally different from existing ones (using only simple Schur complement computations), we believe that this is {of independent interest.} 

\begin{remark}
A MATLAB implementation of \Cref{algo-pers-lap-schur} for unweighted 
simplicial pairs is given in \cite{memoli2021github}. A recent preprint \cite{wang2020hermes} by some of the authors of \cite{wang2020persistent} describes an alternative software implementation of the persistent Laplacian which is available at \cite{wang2020github}.
\end{remark}

\subsection{Relationship with the notion of effective resistance}\label{sec:relation with effective resistance}
Let $K=(V^K,E^K,w^K)$ be a connected weighted graph. Unless otherwise specified, for any weighted graph considered in this section, we assume that $w^K$ satisfies that $w^K_0=w^K|_{S_0^K}\equiv 1$, i.e., the vertices of the graph are unweighted. For any two vertices $v,w\in V^K$, we let $\partial_{[v,w]}\coloneqq-[v]+[w]\in C_0^K$. Let $D_{[v,w]}^K\coloneqq\chi_w-\chi_v\in \R^{n_0^K}$ denote the vector representation of $\partial_{[v,w]}$ in $C_0^K$, where $\chi_v\in \R^{n_0^K}$ is the indicator vector of $v\in V^K$. We consider that each edge $e\in E^K$ has an \emph{electrical conductance} $w^K(e)$. Then, the \emph{effective resistance} $\mathfrak{R}_{v,w}^K$ between $v$ and $w$ is defined by 
\begin{equation}\label{eq:effective resistance graph}
    \mathfrak{R}_{v,w}^K\coloneqq \lc D_{[v,w]}^K\rc^\T\lc\lap_0^K\rc^\dagger D_{[v,w]}^K.
\end{equation}

Given a graph pair $K\hookrightarrow L$, by \Cref{thm:persis-Laplacian-schur-formula} the persistent Laplacian $\lap_0^{K,L}$ turns out to be the graph Laplacian of a new weighted graph.

\begin{proposition}[{\cite[Lemma 2.1]{dorfler2012kron}}]\label{prop:persistent graph is a graph}
Suppose that $K\hookrightarrow L$ is a graph pair. Assume that $L$ is connected and $w^L_0\equiv 1$. Then, $\lap_{0}^{K,L}=\lap_{0,\mathrm{up}}^{K,L}$ is the graph Laplacian $\lap_0^{\tilde{K}}$ of a connected weighted graph $\tilde{K}=\lc V^{\tilde{K}},E^{\tilde{K}},w^{\tilde{K}}\rc$ such that $V^{\tilde{K}}=V^K$.
\end{proposition}

$\tilde{K}$ is known as the \emph{Kron reduction} of $L$ and $\lap_0^{\tilde{K}}$ is called the \emph{Kron-reduced matrix}. The Kron reduction \cite{kron1939tensor} has been used in network circuit theory, {and it preserves effective resistance (cf. \cite[Theorem 3.8]{dorfler2012kron}). } 
This in turn implies that the persistent Laplacian $\lap_0^{K,L}$ 
is able to recover the  effective resistance $\mathfrak{R}_{v,w}^L$ w.r.t. the larger graph $L$ for all pairs of vertices $v, w\in K$. {The result below follows from \Cref{thm:persis-Laplacian-schur-formula} and \cite[Theorem 3.8]{dorfler2012kron}.}

\begin{theorem}\label{thm:effective-persistent-preserve}
Let $K\hookrightarrow L$ be a graph pair where $L$ is connected. Let $\tilde{K}=(V^K,E^{\tilde{K}},w^{\tilde{K}})$ denote the weighted graph such that $\lap_0^{\tilde{K}}=\lap_0^{K,L}$. Then, $\tilde{K}$ is connected and for two distinct vertices $v,w\in V^K$, we have that $\mathfrak{R}_{v,w}^L=\mathfrak{R}_{v,w}^{\tilde{K}}.$
\end{theorem}

\begin{remark}[Higher dimensional generalization]
The effective resistance has been generalized to the case of simplicial complexes in \cite{kook2018simplicial}. In \Cref{sec:kron reduction}, we show a higher-dimensional extension of \Cref{thm:effective-persistent-preserve}, i.e., that higher dimensional effective resistances are preserved by the up persistent Laplacian, the proof of which deals with the subtleties of the Moore-Penrose generalized inverse directly without resorting to a limiting argument as in the proof of \cite[Theorem 3.8]{dorfler2012kron}. In addition, we provide an example illustrating the impossibility of a higher dimensional generalization of the Kron reduction in the current simplicial setting. 
\end{remark}

The following result controls the change of degrees after applying the Kron reduction. 

\begin{proposition}\label{prop:persistent degree}
Let $\tilde{K}$ be the graph described in \Cref{prop:persistent graph is a graph}. Then, for any $v\in V^K=V^{\tilde{K}}$, we have that $\mathrm{deg}^{\tilde{K}}(v)\leq \mathrm{deg}^L(v)$, where $\mathrm{deg}^G(v)\coloneqq\sum_{w\in G}w^G(\{v,w\})$ is the weighted degree of a vertex in a graph $G$.
\end{proposition}
\begin{proof}
We first observe that $\mathrm{deg}^G(v)\coloneqq\sum_{w\in G}w^G(\{v,w\})=\lc\chi_v^G\rc^\mathrm{T}\lap_0^G \chi_v$. Therefore,
\begin{align*}
    \mathrm{deg}^{\tilde{K}}(v)&=\lc\chi_v^K\rc^\T \lap_0^{\tilde{K}} \chi_v^{\tilde{K}}=\lc\chi_v^{\tilde{K}}\rc^\T \lap_0^{K,L} \chi_v^{\tilde{K}}=\lc\chi_v^{\tilde{K}}\rc^\T \lap_0^{L}/\lap_0^{L}(I_K^L,I_K^L) \chi_v^{\tilde{K}}\\
    &=\lc\chi_v^{\tilde{K}}\rc^\T \lap_0^{L}\lc[n_0^K],[n_0^K]\rc\chi_v^{\tilde{K}}-\lc\chi_v^{\tilde{K}}\rc^\T\lap_0^L\lc[n_0^K],I_K^L\rc\lc\lap_0^{L}\lc I_K^L,I_K^L\rc\rc^\dagger\lap_0^L\lc I_K^L,[n_0^K]\rc \chi_v^{\tilde{K}}\\
    &\leq \lc\chi_v^{\tilde{K}}\rc^\T \lap_0^{L}\lc[n_0^K],[n_0^K]\rc\chi_v^{\tilde{K}}= \lc\chi_v^L\rc^\T \lap_0^L \chi_v^L=\mathrm{deg}^L(v)
\end{align*}
\end{proof}

In the case when $K$ consists of only two points in $L$, we have the following explicit relation between the persistent Laplacian and the effective resistance.
\begin{corollary}\label{coro:2point-persistent}
Let $L$ be a connected graph and let $K$ be a two-vertex subgraph with vertex set $V^K=\{v,w\}$. Then, $\lap_0^{K,L}=\begin{pmatrix}
\frac{1}{\mathfrak{R}_{v,w}^L} &  -\frac{1}{\mathfrak{R}_{v,w}^L}\\
-\frac{1}{\mathfrak{R}_{v,w}^L} & \frac{1}{\mathfrak{R}_{v,w}^L} 
\end{pmatrix}$.
\end{corollary}
\begin{proof}[Proof of \Cref{coro:2point-persistent}]
By \Cref{thm:effective-persistent-preserve} (or by \cite[Lemma 3.10]{dorfler2012kron}), it is easy to show that 
$\lc\lap_0^{K,L}\rc^\dagger=\begin{pmatrix}
\frac{\mathfrak{R}_{v,w}^L}{4} &  -\frac{\mathfrak{R}_{v,w}^L}{4}\\
-\frac{\mathfrak{R}_{v,w}^L}{4} & \frac{\mathfrak{R}_{v,w}^L}{4} 
\end{pmatrix}.$ Therefore, $\lap_0^{K,L}=\begin{pmatrix}
\frac{1}{\mathfrak{R}_{v,w}^L} &  -\frac{1}{\mathfrak{R}_{v,w}^L}\\
-\frac{1}{\mathfrak{R}_{v,w}^L} & \frac{1}{\mathfrak{R}_{v,w}^L} 
\end{pmatrix}.$
\end{proof}

\subsubsection{Effective resistance between disjoint sets}
The effective resistance between two vertices has been generalized to the case of two disjoint sets of vertices in \cite[Exercise 2.13]{lyons2017probability} via an energy minimization process. In \cite{song2019extension}, a formula invoking the graph Laplacian was used to define the effective resistance between disjoint sets. The two definitions are equivalent (see \Cref{sec: energy formulation} for a proof) and in this section we adopt the definition from \cite{song2019extension}.

\begin{figure}
    \centering
    \includegraphics[width=0.7\linewidth]{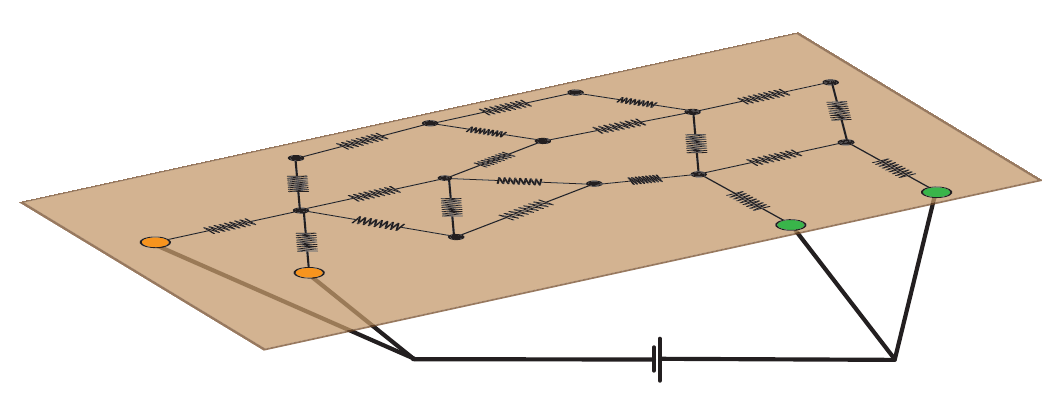}
    \caption{The effective resistance between two sets of vertices from the circuit theory perspective. One set is represented by the orange nodes, the other onbe by the green nodes. See \Cref{sec: energy formulation} for more details.}
    \label{fig:res-sets}
\end{figure}

Let $K$ be a connected weighted graph. For any non-empty disjoint subsets $A,B\subseteq V^K$, let $V^J\coloneqq A\cup B$ and let $J$ be the induced subgraph with vertex set $V^J$. Then, following \cite{song2019extension}, the effective resistance $\mathfrak{R}^K_{A,B}$ between $A$ and $B$ is defined as follows
\begin{equation}\label{eq:resistance btw sets}
    \mathfrak{R}^K_{A,B}\coloneqq \lc\lc\chi_A^{J}\rc^\mathrm{T}\cdot \lap_0^{K}/\lap_0^K\lc V^J,V^J\rc\cdot\chi_A^J\rc^{-1},
\end{equation}
where $\lap_0^K\lc V^J,V^J\rc$\footnote{$\lap_0^K\lc V^J,V^J\rc$ was required to be non-singular in \cite{song2019extension}. This holds automatically as long as $K$ is connected; see \cite[Lemma 2.1]{dorfler2012kron}.} denotes the submatrix of $\lap_0^K$ with rows and columns indexed by $V^J$ and $\chi_A^J\in\R^{n_0^J}$ denotes the indicator vector of $A\subseteq V^J$. By \Cref{thm:persis-Laplacian-schur-formula}, we have that $\mathfrak{R}^K_{A,B}=\lc\lc\chi_A^{J}\rc^\mathrm{T}\cdot \lap_0^{J,K}\cdot\chi_A^J\rc^{-1}$. In particular, when $A\cup B=V^J=V^K$, $\mathfrak{R}^K_{A,B}= \lc\lc\chi_A^K\rc^\mathrm{T}\cdot \lap_0^K\cdot\chi_A^K\rc^{-1}$. We call $\mathfrak{C}^K_{A,B}\coloneqq \frac{1}{\mathfrak{R}^K_{A,B}}$ the \emph{effective conductance} between $A$ and $B$.

\begin{remark}\label{rmk:resistance between sets}
{Note that: (a) When $A=\{v\}$ and $B=\{w\}$ are two singleton sets, it is easy to see that $\mathfrak{R}^K_{A,B}=\mathfrak{R}^K_{v,w}$. (b) \Cref{eq:resistance btw sets} might seem asymmetric with respect to $A$ and $B$. In fact, we have that $\mathfrak{R}^K_{A,B}=- \lc\lc\chi_A^{J}\rc^\mathrm{T}\cdot \lap_0^{J,K}\cdot\chi_B^J\rc^{-1}$; see \cite[Lemma 3]{song2019extension}.}  (c) An explanation from the point of view of circuit theory is given in \Cref{sec:effective resistance between sets}; see Figure \ref{fig:res-sets}.
\end{remark}

As a generalization of \Cref{thm:effective-persistent-preserve}, we establish the following result:
\begin{theorem}\label{thm:preserve of set resistance}
For a graph pair $K\hookrightarrow L$ where $L$ is connected, let $\tilde{K}=\lc V^K,E^{\tilde{K}},w^{\tilde{K}}\rc$ denote the graph such that $\lap_0^{\tilde{K}}=\lap_0^{K,L}$. Then $\mathfrak{R}^{\tilde{K}}_{A,B}=\mathfrak{R}^L_{A,B}$ for any disjoint $A,B\subseteq V^K$.
\end{theorem}

\begin{proof}
By \Cref{thm:effective-persistent-preserve}, $\tilde{K}$ is a connected graph and thus $\mathfrak{R}^{\tilde{K}}_{A,B}$ is well defined. Then, let $V^J\coloneqq A\cup B\subseteq V^K\subseteq V^L$ and let $J$ denote the induced graph in $L$ with vertex set $V^J$. By \Cref{lm:quotient formula}, we have that $\lap_0^{J,\tilde{K}}=\lap_0^{J,L}$. Then, by \Cref{eq:resistance btw sets} we have that $\mathfrak{R}^{\tilde{K}}_{A,B}=\mathfrak{R}^L_{A,B}.$
\end{proof}


\subsection{Persistent Cheeger inequality for graph pairs $K\hookrightarrow L$}\label{sec:cheeger}
The \emph{Cheeger constant} \cite{chung1996laplacians} $h^K$ of a weighted graph $K=(V^K,E^K,w^K)$ is defined as follows: 
\[h^K\coloneqq \min_{\substack{{\emptyset\neq A\subsetneq V^K}\\{|A|\leq \frac{1}{2}|V^K|}}}\frac{\left\|E^K(A,V^K\backslash A)\right\|_{w^K}}{|A|},\]
where $E^K(A,B)$ denotes the set of all edges $\{v,w\}\in E^K$ such that $v\in A$ and $w\in B$, $|A|$ denotes the cardinality of $A$ and $\left\|E^K(A,B)\right\|_{w^K}\coloneqq\sum_{\{v,w\}\in E^K(A,B)}w^K(\{v,w\})$. 

The Cheeger constant $h^K$ measures the edge expansion \cite{hoory2006expander} of $K$ and it is related to the second smallest eigenvalue $\lambda_{0,2}^K$ of the graph Laplacian $\Delta_0^K$ as follows:
\begin{equation}\label{eq:cheeger inequality}
 \frac{\left(h^K\right)^2}{2\,d_{\mathrm{max}}^K}\leq \lambda_{0,2}^K\leq 2\,h^K,  
\end{equation}
where $d_\mathrm{max}^K\coloneqq\max_{v\in V^K}\deg^K(v)$. \Cref{eq:cheeger inequality} is called the \emph{discrete Cheeger inequality} \cite{chung1996laplacians,gundert2015higher,keller2016general}, which is a discrete analogue to isoperimetric inequalities in Riemannian geometry \cite{buser1982note,cheeger1969lower}.

In this section, we define a \emph{persistent Cheeger constant} for any graph pair $K\hookrightarrow L$ via the effective resistance and establish a corresponding \emph{persistent Cheeger inequality} in analogy to \Cref{eq:cheeger inequality}.

To this end,  for a subset $\emptyset\neq A\subsetneq V^K$, we first observe the following relationship between $\left\|E^K(A,V^K\backslash A)\right\|_{w^K}$ and the effective conductance $\mathfrak{C}^K_{A,V^K\backslash A}$ in a given weighted graph $K$:

\begin{lemma}\label{lm:E_k=R_k}
Given a weighted graph $K$ and any $\emptyset\neq A\subsetneq V^K$, we have that
\[\left\|E^K(A,V^K\backslash A)\right\|_{w^K}=\mathfrak{C}^K_{A,V^K\backslash A}.\]
\end{lemma}

\begin{proof}
By \Cref{rmk:resistance between sets}, 
\[\mathfrak{C}^K_{A,V^K\backslash A}=- \lc\chi_A^{K}\rc^\mathrm{T}\cdot \lap_0^{K}\cdot\chi_{V^K\backslash A}^K=\sum_{v\in A}\sum_{\substack{{w\in V^K\backslash A}\\{\{v,w\}\in S_1^K}}}w^K(\{v,w\})=\left\|E^K(A,V^K\backslash A)\right\|_{w^K}\]
\end{proof}

Hence, the Cheeger constant of a weighted graph $K$ can  be equivalently expressed as 
\begin{equation}\label{eq:new cheeger}
    h^K= \min_{\substack{{\emptyset\neq A\subsetneq V^K}\\{|A|\leq \frac{1}{2}|V^K|}}}\frac{\mathfrak{C}^K_{A,V^K\backslash A}}{|A|}.
\end{equation}
We will use this expression to generalize the Cheeger constant to the case of graph pairs.
In the case of a graph pair $K\hookrightarrow L$, we define a persistent Cheeger constant by replacing the right hand side of \Cref{eq:new cheeger} with the effective conductance between subsets of vertices of $K$ \emph{inside the ambient graph $L$}:

\begin{definition}[Persistent Cheeger constant]
The \emph{persistent Cheeger constant} $h^{K,L}$ for a graph pair $K\hookrightarrow L$ is defined as follows:
\[h^{K,L}\coloneqq\min_{\substack{{\emptyset\neq A\subsetneq V^K}\\{|A|\leq \frac{1}{2}|V^K|}}}\frac{\mathfrak{C}^L_{A,V^K\backslash A}}{|A|}. \]
\end{definition}

It is clear that when $K=L$, $h^{K,L}$ reduces to $h^K$. {The following result indicates the persistent Cheeger constant grows as the ambient graph becomes ``more connected'':}
\begin{proposition}\label{prop:cheeger constant ineq}
Consider three weighted graphs $K\subseteq L_1\subseteq L_2$. Then,
\[h^K\leq h^{K,L_1}\leq h^{K,L_2}.\]
\end{proposition}

See {\Cref{sec:combinatorial upper bd} for comments about using other possible generalizations of the standard Cheeger constant to the case of graph pairs.}

\begin{remark}[Probabilistic interpretation] 
Consider the canonical random walk $\{X_n\}_{n=0}^\infty$ defined on $L$ with $V^L$ being the set of states and the transition probability from $v$ to one of its neighbors $w$ is $\frac{w^L(\{v,w\})}{\deg^L(v)}$. For any $A\subsetneq V^K$, let $B\coloneqq V^K\backslash A$. We establish in \Cref{sec:relation random walk} that $\mathfrak{C}^L_{A,B}$ is proportional to the \textit{escape probability from $A$ to $B$}, i.e., the probability of the walk, starting randomly from a vertex in $A$, reaches $B$ before returning to $A$. In this way, we see that $\mathfrak{C}^L_{A,B}$ measures whether $A$ and $B$ are well-separated in $L$, i.e., the larger $\mathfrak{C}^L_{A,B}$ is, the more connected $A$ and $B$ are. Thus, $h^{K,L}$ measures the capability of $K$ being partitioned into two well-separated parts in $L$.
\end{remark}


Our definition of persistent Cheeger constant is handy to deal with, and we hence arrive at the following \emph{persistent} Cheeger inequality.

\begin{theorem}[Persistent Cheeger inequality]\label{thm:persistent Cheeger inequality}
Let $K\hookrightarrow L$ be a weighted graph pair, then 
\begin{equation}\label{eq:persistent Cheeger ineq}
    \frac{\left(h^{K,L}\right)^2}{2\,d_{\mathrm{max}}^{K,L}}\leq\lambda^{K,L}_{0,2}\leq 2\,h^{K,L},
\end{equation}
where $d_\mathrm{max}^{K,L}\coloneqq\max_{v\in V^K}\deg^L(v)$ and $\lambda_{0,2}^{K,L}$ denotes the second smallest eigenvalue of $\Delta_0^{K,L}$.
\end{theorem}

Note that when $K=L$, \Cref{eq:persistent Cheeger ineq} reduces to \Cref{eq:cheeger inequality}. So our persistent Cheeger inequality is a proper generalization of the standard discrete Cheeger inequality.

\begin{proof}
By \Cref{prop:persistent graph is a graph}, $\lap_0^{K,L}$ is the graph Laplacian $\lap_0^{\tilde{K}}$ of a weighted graph $\tilde{K}=(V^K,E^{\tilde{K}},w^{\tilde{K}})$, so  that $\lambda^{K,L}_{0,2}=\lambda_{0,2}^{\tilde{K}} $. By \Cref{eq:cheeger inequality}, we have that $ \frac{(h^{\tilde{K}})^2}{2d_{\mathrm{max}}^{\tilde{K}}}\leq \lambda^{K,L}_{0,2}\leq 2h^{\tilde{K}}.$
Note that by \Cref{lm:E_k=R_k}
\[h^{\tilde{K}}=\min_{\substack{{\emptyset\neq A\subsetneq V^K}\\{|A|\leq \frac{1}{2}|V^K|}}}\frac{\left\|E^K(A,V^K\backslash A)\right\|_{w^K}}{|A|}=\min_{\substack{{\emptyset\neq A\subsetneq V^K}\\{|A|\leq \frac{1}{2}|V^K|}}}\frac{\mathfrak{C}^{\tilde{K}}_{A,V^K\backslash A}}{|A|},\]
where we have used the fact that $V^K=V^{\tilde{K}}$. By \Cref{thm:preserve of set resistance}, we have that $\mathfrak{R}^{\tilde{K}}_{A,V^K\backslash A}=\mathfrak{R}^L_{A,V^K\backslash A}$ and thus $\mathfrak{C}^{\tilde{K}}_{A,V^K\backslash A}=\mathfrak{C}^L_{A,V^K\backslash A}$. This implies that $h^{\tilde{K}}=h^{K,L}$. 

For any $v\in V^K=V^{\tilde{K}}$, by \Cref{prop:persistent degree} we have that
\[\sum_{w\in V^K}w^{\tilde{K}}(\{v,w\})\leq \sum_{w\in V^L}w^{L}(\{v,w\})\]
and thus $d_{\mathrm{max}}^{\tilde{K}}\leq d_\mathrm{max}^{K,L}$. Therefore,
\[\frac{(h^{K,L})^2}{2d_{\mathrm{max}}^{K,L}}\leq \frac{(h^{K,L})^2}{2d_{\mathrm{max}}^{\tilde{K}}}=\frac{(h^{\tilde{K}})^2}{2d_{\mathrm{max}}^{\tilde{K}}}\leq \lambda^{K,L}_{0,2}\leq 2h^{\tilde{K}}= 2h^{K,L}.\]
\end{proof}

\subsubsection{A combinatorial upper bound for $\lambda_{0,2}^{K,L}$}\label{sec:combinatorial upper bd}
When graphs are unweighted, we provide a combinatorial upper bound for $\lambda_{0,2}^{K,L}$. 

A \emph{path} in a graph $K=(V^K,E^K)$ is a tuple $p=(v_0,\ldots,v_n)$ such that $v_i\in V^K$ for each $i=0,\ldots,n$ and $\{v_i,v_{i+1}\}\in E^K$ for each $i=0,\ldots,n-1$. For two nonempty disjoint subsets $A,B \subseteq V^K$, we denote by $P_K(A,B)$ the set of all paths $p=(v_0,\ldots,v_n)$ in $K$ satisfying: {(i) $v_0\in A,v_n\in B$ and $v_i\notin A\cup B$ for $i=1,\ldots,n-1$; (ii) $\{v_i,v_{i+1}\}\neq \{v_j,v_{j+1}\}$ for $i\neq j$. If $A=\{v\}$ and $B=\{w\}$ are one-point sets, then we also denote $P_K(v,w)\coloneqq P_K(\{v\},\{w\})$.} The following Nash-Williams inequality \cite[Lemma 2.1]{lyons2020induced} permits relating $P_K(A,B)$ with $\mathfrak{R}^K_{A,B}$. 
\begin{lemma}[Nash-Williams inequality]\label{lm:nash williams}
Let $K$ be a weighted graph. Let $A,B$ be nonempty disjoint subsets of $V^K$. A set $\Pi\subseteq E^K$ is called a \emph{cut set} between $A$ and $B$ if for any $v\in A$ and $w\in B$, every path from $v$ to $w$ contains an edge in $\Pi$. Suppose $\Pi_1,\ldots,\Pi_n$ are disjoint cut sets between $A$ and $B$. Then,
\[\mathfrak{R}^K_{A,B}\geq \sum_{k=1}^n\lc\sum_{e\in \Pi_k}w^K(e)\rc^{-1}.\]
\end{lemma}

Now, consider a graph pair $K\hookrightarrow V$. Let $\emptyset\neq A\subseteq V^K$ and let $B\coloneqq V^K\backslash A$. Then, let $p_1,\ldots,p_N$ denote all the paths in $P_L(A,B)$. Choose an arbitrary edge $e_i$ from each path $p_i$. The set $\Pi\coloneqq\{e_i:i=1,\ldots,N\}$ is obviously a cut set between $A$ and $B$. By \Cref{lm:nash williams} we have that $\mathfrak{C}^L_{A,B}\leq \sum_{e\in \Pi}w^L(e)\leq |P_L(A,B)|$. By \Cref{thm:persistent Cheeger inequality} we have the following upper bound for $\lambda_{0,2}^{K,L}$ which arises by minimizing the number of paths in $L$ connecting the two sets in a bipartition of $V^K$: 
\[\frac{1}{2}\lambda_{0,2}^{K,L}\leq \,h^{K,L}\leq \min_{\substack{{\emptyset\neq A\subsetneq V^K}\\{|A|\leq \frac{1}{2}|V^K|}}}\frac{\left|P_L(A,V^K\backslash A)\right|}{|A|}=:{h}^{K,L}_\mathrm{path}.\]

{A priori, it seems plausible that one could have used the right hand side of the above inequality, ${h}_\mathrm{path}^{K,L}$ as the definition of the persistent Cheeger constant. However, as we show in \Cref{sec:alt-pcc}, this quantity does not have a good interplay with the second persistent eigenvalue, i.e., ${h}_\mathrm{path}^{K,L}$ cannot be upper bounded by $\lambda_{0,2}^{K,L}$ in any suitable sense.}

\section{The persistent Laplacian for simplicial filtrations}\label{sec:filtration}
We now extend the setting of \Cref{sec:pairs} for simplicial pairs to a simplicial filtration. 
\subsection{Formulation}
Let $\mathbf{K} = \{K_t\}_{t\in T}$ be a simplicial filtration with an index set $T\subseteq \R$. For each $t\in T$ and $q\in\mathbb{N}$ we let $C_q^t\coloneqq C_q^{K_t}$, $S_q^t\coloneqq S_q^{K_t}$ and $w_q^t\coloneqq w_q^{K_t}$. For $s\leq t\in T$ we let   
\[C_q^{t,s}\coloneqq\left\{c\in C_q^t:\,\partial_q^t(c)\in C_{q-1}^s\right\}\subseteq C_{q}^t.\] 
Let $\partial_q^{t,s}$ be the restriction of $\partial_q^t$ to $C_q^{t,s}$. Then, $\partial_q^{t,s}$ is a map from $C_q^{t,s}$ to $C_{q-1}^s$. Finally, we define the $q$-th persistent Laplacian 
$\Delta_q^{s,t}: C_q^{s}\rightarrow{C_q^s}$
by
\begin{equation}\label{eq:persistent laplacian filtration}
    \Delta_q^{s,t}\coloneqq \underbrace{\partial_{q+1}^{t,s}\circ \left(\partial_{q+1}^{t,s}\right)^\ast}_{\Delta_{q,\mathrm{up}}^{s,t}}+\underbrace{\left(\partial_q^s\right)^\ast\circ\partial_q^s}_{\Delta_{q,\mathrm{down}}^s},
\end{equation}
where we view $C_q^t$ for each $t\in T$ as a Hilbert space with the inner product $\langle\cdot,\cdot\rangle_{w_q^t}$ and $A^\ast$ means the adjoint of an operator $A$ under these inner products. 
We also let $\Delta_q^t$ denote the $q$-th Laplacian of $K_t$ for $t\in T$. Note that $\Delta_q^{t,t}=\Delta_q^t$ (cf. \Cref{ex:trivial cases}).

\subsection{An algorithm for $\lap_q^{s,t}$} 

\newcommand{\myGamma}  {{\Gamma}}
\newcommand{\newD} {\lap_{q,\mathrm{up}}}

Consider the simplicial filtration $K_1 \hookrightarrow K_2 \hookrightarrow \cdots \hookrightarrow K_m$ where each $K_{t+1}$ contains exactly one more simplex than $K_t$ for $t=1,\ldots,m-1$. In this section, we show that, for a fixed index $t\in[m]$, we can compute the matrix representation $\lap_q^{s,t}$ of the persistent Laplacian $\Delta_q^{s,t}$, for all $1\leq s\leq t$, in time $O\lc t \lc n_q^t\rc^2+n_{q+1}^t\rc$, where $n_q^t\coloneqq n_q^{K_t}$ is the number of $q$-simplices in $K_t$. Note that this is more efficient than applying the Schur complement formula for $\lap_q^{s,t}$ (\Cref{eq:schur complement up laplacian}) $t$ times, which will lead to $O\left(t\lc n_q^t\rc^3+t\,n_{q+1}^t\right)$ total time. This result is again achieved via the relation between persistent Laplacian with Schur complement (cf. \Cref{thm:persis-Laplacian-schur-formula}).  

Recall from \Cref{eq:persistent laplacian filtration} that for any $1\le s \le t$, 
$\lap_q^{s,t} = \lap_{q, \mathrm{up}}^{s,t} + \lap_{q,\mathrm{down}}^s$. 
Since $\lap_{q,\mathrm{down}}^s$ can be constructed in time $O\lc\lc n_q^s\rc^2\rc = O\lc \lc n_q^t\rc^2\rc$ (cf. \Cref{sec:computation of up and down}), the set of $\lap_{q,\mathrm{down}}^s$ for all $1\le s \le t$ can be computed in $O\lc t \lc n_q^t\rc^2\rc$ time.

For simplicity, we assume that $S_q^s = \{\sigma_1, \ldots, \sigma_s\}$ for each $s=1,\ldots,t$, that is, $K_{s+1}$ contains exactly one more $q$-simplex than $K_s$ for $s=1,\ldots,t-1$. It then follows that $\newD^{s,t} = \newD^{t} / \newD^{t}(I_s^t, I_s^t)$, where $I_s^t$ is the index set $I_s^t= \{s+1,s+2,\ldots,t\}$. By \Cref{rmk:w=1 weighted laplacian} and \Cref{lm:psd-proper}, $\newD^{t}(I_s^t, I_s^t)$ is proper in $\newD^{t}$ for each $s=1,\ldots,t-1$. Therefore, following the Quotient Formula (\Cref{lm:quotient formula}), to compute $\newD^{s,t}$, one can perform an iterative reduction from $\newD^{t}$ to $\newD^{t-1, t}, \ldots, \newD^{s+1, t}$, and down to $\newD^{s,t}$. 
More precisely, for any $\ell \le t,$
\begin{align}\label{eqn:iterKron}
    \newD^{\ell-1,t}(i,j) &=\begin{cases}\newD^{\ell, t}(i,j) -  \frac{\newD^{\ell,t}(i,\ell) \newD^{\ell,t}(j, \ell)}{\newD^{\ell,t}(\ell,\ell)},& \text{if }\newD^{\ell,t}(\ell,\ell)\neq 0\\
    \newD^{\ell,t}(i,j),& \text{if }\newD^{\ell,t}(\ell,\ell)= 0\end{cases}  ~~~\text{for any}~i,j \in [\ell-1]. 
\end{align}
\Cref{eqn:iterKron} reduces to the celebrated Kron reduction formula (see Equation (16) of \cite{dorfler2012kron}) when $K_t$ is a connected graph, $q=0$ and $w_0^t\equiv 1$. In other words, $\newD^{\ell-1,t}$ can be computed from $\newD^{\ell,t}$ in time linear to the size of the matrix, which is bounded by $O\lc \lc n_q^t\rc^2\rc$. Note that from \Cref{sec:computation of up and down} we know computing $\newD^t$ takes time $O\lc n_{q+1}^t\rc$. It then follows that using \Cref{eqn:iterKron}, we can compute $\newD^{s,t}$, for all $1\leq s\leq t$ iteratively in $O\lc t \lc n_q^t\rc^2+n_{q+1}^t\rc$ total time. We summarize our discussion into the following theorem.

\begin{theorem}
Let $K_1 \hookrightarrow \cdots \hookrightarrow K_m$ be a simplicial filtration where each $K_{t+1}$ contains exactly one more simplex than $K_t$ for all $t\in[m-1]$. For any fixed $t\in[m]$, we can compute the whole set $\left\{\lap_q^{s,t}\right\}_{s=1}^t$ of persistent Laplacians in $O\lc t \lc n_q^t\rc^2+n_{q+1}^t\rc$ time. This also implies that we can compute all $\lap_q^{i, j}$, for any $1\leq i\le j \leq m$, in $O\lc m^2 \lc n_q^m\rc^2+m\,n_{q+1}^m\rc$ total time.
\end{theorem}

\subsection{Monotonicity and stability of persistent eigenvalues}
Recall from \Cref{sec:basic properties of persistent Laplacian} that for a simplicial pair $K\hookrightarrow L$, $\lambda_{q,k}^{K,L}$ denotes the $k$-th smallest eigenvalue of $\Delta_q^{K,L}$. Now, given a simplicial filtration $\mathbf{K} = \{K_t\}_{t\in T}$, we define its \emph{$k$-th persistent eigenvalue} $\lambda_{q,k}^{s,t}(\mathbf{K})$ for each $s\leq t\in T$ by $\lambda_{q,k}^{s,t}(\mathbf{K})\coloneqq\lambda_{q,k}^{K_s,K_t}$. We define the \emph{$k$-th up-persistent eigenvalue} $\lambda_{q,\mathrm{up},k}^{s,t}(\mathbf{K})$ for each $s\leq t\in T$ to be the $k$-th smallest eigenvalue of $\Delta_{q,\mathrm{up}}^{s,t}$. Whenever the underlying filtration $\mathbf{K}$ is clear from the context, we let $\lambda_{q,k}^{s,t}\coloneqq\lambda_{q,k}^{s,t}(\mathbf{K})$ and $\lambda_{q,\mathrm{up},k}^{s,t}\coloneqq\lambda_{q,\mathrm{up},k}^{s,t}(\mathbf{K})$.

In \cite{wang2020persistent} the authors suggest that invariants similar to persistent eigenvalues could be useful for shape classification applications. With that in mind,  we now explore both their monotonicity and stability properties, concluding with \cref{thm:stab-evals}.

\begin{theorem}[Monotonicity of up persistent Laplacian eigenvalues]\label{thm:monotonicity up eigen}
Let $\mathbf{K} = \{K_t\}_{t\in T}$ be a simplicial filtration and let $q\in\mathbb{N}$. Then, for any $t_1\leq t_2\leq t_3\in T$, we have for each $k=1,\ldots,n_q^{t_1}$ that $\lambda_{q,\mathrm{up},k}^{t_1,t_2}\leq \lambda_{q,\mathrm{up},k}^{t_1,t_3}$ and $\lambda_{q,\mathrm{up},k}^{t_2,t_3}\leq \lambda_{q,\mathrm{up},k}^{t_1,t_3}$.
\end{theorem}
The proof exploits the connection of the up-Laplacian with Schur complements (\Cref{thm:persis-Laplacian-schur-formula}).

\begin{proof}
By the min-max theorem (see for example \cite[Theorem 2.1]{horak2013spectra}), we have for any $s\leq t\in T$ and for each $k=1,\ldots,n_q^{t_1}$ that
\[\lambda_{q,\mathrm{up},k}^{s,t}=\min_{V_k\subseteq C_q^{s}}\max_{g\in V_k}\frac{\left\langle\Delta_{q,\mathrm{up}}^{s,t}g,g\right\rangle_{w_q^{s}}}{\langle g,g\rangle_{w_q^{s}}},\]
where the minimum is taken over all $k$-dim subspaces $V_k$ of $C_q^{s}$. Then, in order to prove that $\lambda_{q,\mathrm{up},k}^{t_1,t_2}\leq \lambda_{q,\mathrm{up},k}^{t_1,t_3}$, we only need to verify that $\langle\Delta_{q,\mathrm{up}}^{t_1,t_2}g,g\rangle_{w_q^{t_1}}\leq\langle\Delta_{q,\mathrm{up}}^{t_1,t_3}g,g\rangle_{w_q^{t_1}}$ for any $k$-dim subspace $V_k\subseteq C_q^{t_1}$ and any $g\in V_k$.

Now, since $C_{q+1}^{t_2,t_1}\subseteq C_{q+1}^{t_3,t_1}$, we consider an orthogonal decomposition $C_{q+1}^{t_3,t_1}=C_{q+1}^{t_2,t_1}\bigoplus \left(C_{q+1}^{t_2,t_1}\right)^\perp$. Then, we have the decomposition $\partial_{q+1}^{t_3,t_1}=\partial_{q+1}^{t_2,t_1}\oplus \partial^\perp$, where $\partial^\perp$ maps $\left(C_{q+1}^{t_2,t_1}\right)^\perp$ into $C_q^{t_1}$.
Therefore, we have that
\begin{align*}
    \Delta_{q,\mathrm{up}}^{t_1,t_3}={\partial_{q+1}^{t_3,t_1} \left(\partial_{q+1}^{t_3,t_1}\right)^*}={\partial_{q+1}^{t_2,t_1} \left(\partial_{q+1}^{t_2,t_1}\right)^*}+\partial^\perp \left(\partial^\perp\right)^*=\Delta_{q,\mathrm{up}}^{t_1,t_2}+\partial^\perp \left(\partial^\perp\right)^*.
\end{align*}
This implies the following and thus $\lambda_{q,\mathrm{up},k}^{t_1,t_2}\leq \lambda_{q,\mathrm{up},k}^{t_1,t_3}$:
\begin{align*}
    \left\langle\Delta_{q,\mathrm{up}}^{t_1,t_3}g,g\right\rangle_{w_q^{t_1}}&=\left\langle\Delta_{q,\mathrm{up}}^{t_1,t_2}g,g\right\rangle_{w_q^{t_1}}+\left\langle \partial^\perp \left(\partial^\perp\right)^*g,g\right\rangle_{w_q^{t_1}}\\
    &=\left\langle\Delta_{q,\mathrm{up}}^{t_1,t_2}g,g\right\rangle_{w_q^{t_1}}+\left\langle  \left(\partial^\perp\right)^*g, \left(\partial^\perp\right)^*g\right\rangle_{w_q^{t_1}}\geq \left\langle\Delta_{q,\mathrm{up}}^{t_1,t_2}g,g\right\rangle_{w_q^{t_1}}.
\end{align*}

As for $\lambda_{q,\mathrm{up},k}^{t_2,t_3}\leq \lambda_{q,\mathrm{up},k}^{t_1,t_3}$, we will apply \Cref{thm:persis-Laplacian-schur-formula}. For notational simplicity, we let $I_s^t\coloneqq [n_q^t]\backslash [n_q^s]$. Since the matrix $\lap_{q,\mathrm{up}}^{t_3}$ is positive semi-definite, both $\lap_{q,\mathrm{up}}^{t_3}\lc I_{t_2}^{t_3}, I_{t_2}^{t_3}\rc$ and $\lap_{q,\mathrm{up}}^{t_3}\lc I_{t_1}^{t_3}, I_{t_1}^{t_3}\rc$ are proper in $\lap_{q,\mathrm{up}}^{t_3}$ (cf. \Cref{lm:psd-proper}). Moreover, $\lap_{q,\mathrm{up}}^{t_3}\lc I_{t_2}^{t_3}, I_{t_2}^{t_3}\rc$ is proper in $\lap_{q,\mathrm{up}}^{t_3}\lc I_{t_1}^{t_3}, I_{t_1}^{t_3}\rc$. Then, by \Cref{lm:quotient formula}, $\lap_{q,\mathrm{up}}^{t_3}/\lap_{q,\mathrm{up}}^{t_3}\lc I_{t_1}^{t_3}, I_{t_1}^{t_3}\rc$ is the Schur complement of some proper principal submatrix in $\lap_{q,\mathrm{up}}^{t_3}/\lap_{q,\mathrm{up}}^{t_3}\lc I_{t_2}^{t_3}, I_{t_2}^{t_3}\rc$. By \Cref{lm:psd-proper} and \Cref{lm:eigen-interlacing}, 
\[\lambda_k\lc \lap_{q,\mathrm{up}}^{t_3}/\lap_{q,\mathrm{up}}^{t_3}\lc I_{t_2}^{t_3}, I_{t_2}^{t_3}\rc\rc\leq \lambda_k\lc \lap_{q,\mathrm{up}}^{t_3}/\lap_{q,\mathrm{up}}^{t_3}\lc I_{t_1}^{t_3}, I_{t_1}^{t_3}\rc\rc,\, k = 1,\ldots,n_q^s.\]
Then, by \Cref{thm:persis-Laplacian-schur-formula}, we have that $\lambda_{{q,\mathrm{up}},k}^{t_2,t_3}\leq \lambda_{{q,\mathrm{up}},k}^{t_1,t_3}$ for all $k=1,\ldots,n_q^{t_1}.$
\end{proof}

Note that when $q=0$, $\Delta_0^{s,t}=\Delta_{0,\mathrm{up}}^{s,t}$ for $s\leq t$. Then, we have the following corollary. 

\begin{corollary}
Let $\mathbf{K} = \{K_t\}_{t\in T}$ be a simplicial filtration. Then for any $t_1\leq t_2\leq t_3\in T$, we have for each $k=1,\ldots,n_0^{t_1}$ that $\lambda_{0,k}^{t_1,t_2}\leq \lambda_{0,k}^{t_1,t_3}$ and $\lambda_{0,k}^{t_2,t_3}\leq \lambda_{0,k}^{t_1,t_3}$.
\end{corollary}

A simple adaptation of the proof of the formula $\lambda_{q,\mathrm{up},k}^{t_1,t_2}\leq \lambda_{q,\mathrm{up},k}^{t_1,t_3}$ will give rise to the following monotonicity result for eigenvalues of persistent Laplacians.

\begin{corollary}
Let $\mathbf{K} = \{K_t\}_{t\in T}$ be a simplicial filtration. Given $q\in\mathbb{N}$, then for any $t_1\leq t_2\leq t_3\in T$, we have for each $k=1,\ldots,n_q^{t_1}$ that $\lambda_{q,k}^{t_1,t_2}\leq \lambda_{q,k}^{t_1,t_3}$.
\end{corollary}

\paragraph*{Stability of up-persistent eigenvalues with respect to the interleaving distance} 

\begin{lemma}\label{lm:four point ineq}
Let $K_{t_1}\hookrightarrow K_{t_2}\hookrightarrow K_{t_3}\hookrightarrow K_{t_4}$ be a simplicial filtration over an index set $\{t_1\leq t_2\leq t_3\leq t_4\}$ with
four points. Then, for any $k=1,\ldots, n_q^{t_1}$, we have $\lambda_{{q,\mathrm{up}},k}^{t_1,t_4}\geq \lambda_{{q,\mathrm{up}},k}^{t_2,t_3}$.
\end{lemma}
\begin{proof}
By \Cref{thm:monotonicity up eigen} we have that $\lambda_{{q,\mathrm{up}},k}^{t_1,t_4}\geq \lambda_{{q,\mathrm{up}},k}^{t_1,t_3}$ and $\lambda_{{q,\mathrm{up}},k}^{t_1,t_3}\geq \lambda_{{q,\mathrm{up}},k}^{t_2,t_3}$. Then, $\lambda_{{q,\mathrm{up}},k}^{t_1,t_4}\geq \lambda_{{q,\mathrm{up}},k}^{t_1,t_3}\geq \lambda_{{q,\mathrm{up}},k}^{t_2,t_3}.$
\end{proof}

\begin{definition}[Interleaving distance between simplicial filtrations over $\mathbb{R}$]
Let $\mathbf{K} = \{K_t\}_{t\in \R}$ and $\mathbf{L} = \{L_t\}_{t\in\R}$ be two simplicial filtrations over $\mathbb{R}$  with the same underlying vertex set $V$ and the same index set $\R$. We define the interleaving distance between $\mathbf{K}$ and $\mathbf{L}$ by
\[d_\mathrm{I}^V\lc\mathbf{K},\mathbf{L}\rc\coloneqq\inf\left\{\eps\geq 0:\,\forall t, K_t\subseteq L_{t+\eps}\text{ and }L_t\subseteq K_{t+\eps}\right\},\]
where when we write the inclusion $K\subseteq L$, we implicitly require that $w^{K}=w^{L}|_K$.
\end{definition}

\begin{definition}[Interleaving distance between functions]
Let $\mathbf{Int}$ denote the set of closed intervals in $\R$. Let $f:\mathbf{Int}\rightarrow\R_{\geq 0}$ and $g:\mathbf{Int}\rightarrow\R_{\geq 0}$ be two non-negative functions. We then  define the interleaving distance between $f$ and $g$ by:
\[d_\mathrm{I}\lc f,g\rc\coloneqq\inf\left\{\eps\geq 0:\,\forall I\in\mathbf{Int}, f(I^\eps)\geq g(I)\text{ and }g(I^\eps)\geq f(I)\right\}.\]
\end{definition}
Above, for $\mathbf{Int}\ni I = [a,b]$ and $\eps >0$, we denoted $I^\eps \coloneqq [a-\eps,b+\eps].$

\begin{remark}
The stability theorem given below is structurally similar to claims about stability of the rank invariant, see \cite[Theorem 22]{puuska2020erosion} and  \cite[Remarks 4.10 and 4.11]{kim2020spatiotemporal}. 
\end{remark}

With these definitions we now obtain the following stability theorem:
\begin{theorem}[Stability theorem for up-persistent eigenvalues]\label{thm:stab-evals}
Let $\mathbf{K} = \{K_t\}_{t\in\R}$ and $\mathbf{L} = \{L_t\}_{t\in\R}$ be two simplicial filtrations over the same underlying vertex set $V$. Then,
\begin{equation}\label{eq:stability}
    d_\mathrm{I}\lc\lambda_{{q,\mathrm{up}},k}^{\mathbf{K}},\lambda_{{q,\mathrm{up}},k}^{\mathbf{L}}\rc\leq d_\mathrm{I}^V\lc\mathbf{K},\mathbf{L}\rc,
\end{equation}
where $\lambda_{{q,\mathrm{up}},k}^{\mathbf{K}}:\mathbf{Int}\rightarrow\R_{\geq 0}$ is defined by $\mathbf{Int}\ni I = [a,b]\mapsto \lambda_{{q,\mathrm{up}},k}^{a,b}(\mathbf{K})$.
\end{theorem}
\begin{proof}
If $d_\mathrm{I}^V\lc\mathbf{K},\mathbf{L}\rc=\infty$, then \Cref{eq:stability} holds trivially. Otherwise we assume there exists $\eps\geq 0$ such that $K_t\subseteq L_{t+\eps}$ and $L_t\subseteq K_{t+\eps}$ for all $t\in\R$. For any $I=[a,b]\in\mathbf{Int}$, then $L_{a-\eps}\subseteq K_a\subseteq K_b\subseteq L_{b+\eps}$ is a simplicial filtration related to the following interleaving diagram:
\begin{tikzcd}
&K_a\arrow[hookrightarrow]{r} &  K_b\arrow[hookrightarrow]{dr}&\\
L_{a-\varepsilon}\arrow[hookrightarrow]{ur}\arrow[hookrightarrow]{rrr}&&&L_{b+\varepsilon}\\
\end{tikzcd}. By \Cref{lm:four point ineq}, $\lambda_{{q,\mathrm{up}},k}^{L_{a-\eps},L_{b+\eps}}\geq \lambda_{{q,\mathrm{up}},k}^{K_a,K_b}$. This implies that  $\lambda_{{q,\mathrm{up}},k}^{\mathbf{L}}(I^\eps)\geq\lambda_{{q,\mathrm{up}},k}^{\mathbf{K}}(I)$ for all $I\in\mathbf{Int}$. Similarly, $\lambda_{{q,\mathrm{up}},k}^{\mathbf{K}}(I^\eps)\geq\lambda_{{q,\mathrm{up}},k}^{\mathbf{L}}(I)$ for all $I\in\mathbf{Int}$. Therefore, $d_\mathrm{I}\lc\lambda_{{q,\mathrm{up}},k}^{\mathbf{K}},\lambda_{{q,\mathrm{up}},k}^{\mathbf{L}}\rc\leq \eps$ and thus $d_\mathrm{I}\lc\lambda_{{q,\mathrm{up}},k}^{\mathbf{K}},\lambda_{{q,\mathrm{up}},k}^{\mathbf{L}}\rc\leq d_\mathrm{I}^V\lc\mathbf{K},\mathbf{L}\rc.$
\end{proof}

\section{Discussion}
As a natural progression of the ideas in this paper, where the persistent Laplacian is formulated for inclusion maps, it seems interesting to extend it to the setting of simplicial maps -- a natural extension which would enable other applications such as graph sparsification where clusters of vertices might be collapsed between consecutive levels of a filtration.  

A notion of persistent Laplacian for pairs of manifolds also related by inclusion maps was developed in \cite{chen2017evolutionary}. In the spirit of our paper, it is then natural to attempt to relate the version of the persistent Laplacian from  \cite{chen2017evolutionary} to notions of Schur complement of operators (e.g., \cite{friedrich2018generalized}) in a suitable sense, which may also be related to Poincar\'e-Steklov operators \cite{lebedev1983poincare}.

The Cheeger inequality has both been generalized to higher order (eigenvalues of graph Laplacians) in \cite{lee2014multiway} and to higher dimensional simplicial complexes\cite{steenbergen2014cheeger,gundert2015higher}. This naturally suggests us to consider suitable extensions of our persistent Cheeger inequality to these cases which will provide interpretation of the persistent Laplacian spectrum.

Finally, it is of clear interest to elucidate   stability properties of invariants associated to the persistent Laplacian which generalize the results we proved in Theorem \ref{thm:stab-evals}.

\paragraph{Acknowlegements.} This work is partially supported by National Science Foundation (NSF) under grants CCF-1740761, DMS-1723003, RI-1901360, RI-2050360 and OAC-2039794.

\bibliography{persistLaplacian-bib}
\appendix

\section{Relegated proofs}\label{sec:proofs from the paper}
\begin{proof}[Proof of \Cref{lm:connected component}]
This follows directly from the following obvious observations
\begin{enumerate}
    \item $C_{q-1}^K=\bigoplus_{i=1}^mC_{q-1}^{K_i}$, $C_q^K=\bigoplus_{i=1}^mC_q^{K_i}$ and $C_{q+1}^{L,K}=\bigoplus_{i=1}^mC_{q+1}^{L_i,K_i}$.
    \item $\partial_q^K=\bigoplus_{i=1}^m\partial_q^{K_i}$ and $\partial_{q+1}^{L,K}=\bigoplus_{i=1}^m\partial_{q+1}^{L_i,K_i}$.
\end{enumerate}
\end{proof}

\begin{proof}[Proof of \Cref{thm:connected component}]
For item $1$, let $c^K_0\coloneqq\sum_{v\in S_0^K}w^K_0(v)[v]\in C_0^K$. We prove that $\Delta_0^{K,L}c^K_0=0$ and thus $\lambda_{0,1}^{K,L}=0$.
Set $c^L_0\coloneqq \sum_{v\in S_0^L}w^L_0(v)[v]$. Then,
\[c^L_0=\sum_{v\in S_0^L\backslash {S}_0^K}w^L_0(v)[v]+ \sum_{v\in {S}_0^K}w^K_0(v)[v]=\sum_{v\in S_0^L\backslash {S}_0^K}w^L_0(v)[v]+ c^K_0.\]
For any $c_1\in C_1^{L,K}$, we have the following:
\begin{align*}
    \left\langle \lc\partial_1^{L,K}\rc^*c^K_0,c_1\right\rangle_{w_1^{L,K}}=\left\langle c^K_0,\partial_1^{L,K}c_1\right\rangle_{w_0^{K}}=\left\langle c^L_0,\partial_1^{L,K}c_1\right\rangle_{w_0^{L}}-\left\langle \sum_{v\in S_0^L\backslash {S}_0^K}w^L_0(v)[v],\partial_1^{L,K}c_1\right\rangle_{w_0^{L}},
\end{align*}
where $\langle\cdot,\cdot\rangle_{w_1^{L,K}}$ is the restriction of $\langle\cdot,\cdot\rangle_{w_1^{L}}$ on $C^{L,K}_1$ and we use the fact $w_0^K=w_0^L|_{S_0^K}$ in the rightmost equality.

Since $\partial_1^{L,K}c_1\in C_0^K$, we have that $\left\langle \sum_{v\in S_0^L\backslash {S}_0^K}w^L_0(v)[v],\partial_1^{L,K}c_1\right\rangle_{w_0^{L}}=0$. Now, assume that $c_1=x_1[e_1]+\ldots +x_\ell[e_\ell]$ where each $e_i\in S_1^L$ and $x_i\in\mathbb{R}$. Since $\partial_1^{L,K} [e_i]=\partial_1^L [e_i]=[v_i]-[w_i]$ for some $v_i,w_i\in {S}_0^L$, we have that $\left\langle c_0^L,\partial_1^{L,K} [e_i]\right\rangle_{w_0^L}=0$ for each $i=1,\ldots,\ell$ and thus $\left\langle c^L_0,\partial_1^{L,K} c_1\right\rangle_{w_0^L}=0$. It then follows that 
\[\left\langle \lc\partial_1^{L,K}\rc^*c_0^K,c_1\right\rangle_{w_1^{L,K}}=0,\,\forall c_1\in C_1^{L,K},\]
and thus $\Delta_0^{K,L}c_0^K=\partial_1^{L,K}\lc\partial_1^{L,K}\rc^* c_1 = 0$.

Now, assume that $L$ is connected. Suppose that there exists $0\neq c_0\in C_0^K$ such that $\Delta_0^{L,K}c_0=0$. Then, $\lc\partial_1^{L,K}\rc^*c_0=0$. For any $v,w\in S_0^K$, since $L$ is connected, there exists a $1$-chain $c_1\in C_1^L$ such that $\partial_1^L c_1=[v]-[w]$ (for example, one can take a path in $L$ connecting $v$ and $w$ and let $c_1$ be the corresponding $1$-chain). Then, $c_1\in C_1^{L,K}$ and $\partial_1^{L,K }c_1=[v]-[w]$. Note that,
\[\left\langle c_0,[v]-[w]\right\rangle_{w_0^{K}}=\left\langle c_0,\partial_1^{L,K}c_1\right\rangle_{w_0^{K}}=\left\langle \lc\partial_1^{L,K}\rc^*c_0,c_1\right\rangle_{w_1^{L,K}}=0.\]
This implies that $\left\langle c_0,[v]\right\rangle_{w_0^{K}}=\left\langle c_0,[w]\right\rangle_{w_0^{K}}$ and thus there exists $\alpha\in\mathbb{R}$ such that $\left\langle c_0,[v]\right\rangle_{w_0^{K}}=\alpha$ for each $v\in S_0^K$. Then, $c_0=\alpha\cdot c^K_0$, implying that the multiplicity of $0$ eigenvalue is $1$.

For item $2$, suppose $K$ intersects exactly $m$ connected components of $L$, denoted by $L_1,\ldots,L_m$. Then, by \Cref{lm:connected component} we have that $\Delta_0^{K,L}=\bigoplus_{i=1}^m\Delta_0^{K_i,L_i}$. Then, the spectrum of $\Delta_0^{K,L}$ is the multiset union of the spectra of $\Delta_0^{K_i,L_i}$s. By item 1 and item 2 we have that the multiplicity of zero eigenvalue of $\Delta_0^{K,L}$ is then exactly $m$. 
\end{proof}

\begin{proof}[Proof of \Cref{thm:pers_interior}]
By abuse of the notation, we represent each $c^L\in C_q^L$ by a vector $c^L\in\mathbb{R}^{n_q^L}$. Then, $c^K$ corresponds to the vector $c^K=c^L\lc[n_q^K]\rc\in\mathbb{R}^{n_q^K}$. By \Cref{thm:persis-Laplacian-schur-formula}, the matrix representation $\lap_{q,\mathrm{up}}^{K,L}$ of $\Delta_{q,\mathrm{up}}^{K,L}$ can be computed as follows:
\[\lap_{q,\mathrm{up}}^{K,L}=\lap_{q,\mathrm{up}}^L\lc [n_q^K],[n_q^K]\rc  -\lap_{q,\mathrm{up}}^L\lc [n_q^K],I_K^L\rc\lap_{q,\mathrm{up}}^L\lc I_K^L,I_K^L\rc^\dagger \lap_{q,\mathrm{up}}^L\lc I_K^L,[n_q^K]\rc,\]
where $I_K^L=[n_q^L]\backslash[n_q^K]$. 

Suppose $\sigma_i\in S_q^K$ is an interior simplex, then the $i$-th row of $\lap_{q,\mathrm{up}}^L\lc [n_q^K],I_K^L\rc$ is $0$ (cf. \Cref{sec:computation of up and down}). Then, 
\begin{enumerate}
    \item the $i$-th entry of $\lap_{q,\mathrm{up}}^L\lc [n_q^K],[n_q^K]\rc c^K$ exactly coincides with the $i$-th entry of $\lap_{q,\mathrm{up}}^Lc^L$;
    \item the $i$-th row of $\lap_{q,\mathrm{up}}^L\lc [n_q^K],I_K^L\rc\lap_{q,\mathrm{up}}^L\lc I_K^L,I_K^L\rc^\dagger \lap_{q,\mathrm{up}}^L\lc I_K^L,[n_q^K]\rc$ is $0$.
\end{enumerate}
Therefore, the $i$-th entry of $\lap_{q,\mathrm{up}}^L c^L$ ($=w_q^L(\sigma_i)\left\langle \Delta_{q,\mathrm{up}}^Lc^L, [\sigma_i] \right\rangle_{w_{q}^L}$) agrees with the $i$-th entry of $\lap_{q,\mathrm{up}}^{K,L}c^K$ ($=w_q^K(\sigma_i)\left\langle \Delta_{q,\mathrm{up}}^{K,L}c^K, [\sigma_i] \right\rangle_{w_{q}^K}$). Then by $w_q^K(\sigma_i)=w_q^L(\sigma_i)$, we have that
\[\left\langle \Delta_{q,\mathrm{up}}^Lc^L, [\sigma_i] \right\rangle_{w_{q}^L}=\left\langle \Delta_{q,\mathrm{up}}^{K,L}c^K, [\sigma_i] \right\rangle_{w_{q}^K}. \]
\end{proof}

\begin{proof}[Proof of \Cref{thm:pers-betti-pers-lap}]
{First, we have the following elementary linear algebra fact: The isomorphism follows from \cite[Theorem 5.3]{lim2020hodge} and the equality follows from \cite[Theorem 5.2]{lim2020hodge}.}
\begin{claim}\label{claim:kernel of bb+aa}
Let $A\in\mathbb{R}^{m\times n}$  and let $B\in \mathbb{R}^{n\times p}$. Suppose $AB=0$, then we have
\[\mathrm{ker}(A)/\mathrm{im}(B)\cong \ker(A)\cap \ker\lc B^\T\rc=\mathrm{ker}\lc BB^\T+A^\T A\rc, \]
where $\cong$ denotes isomorphism between vector spaces.
\end{claim}

The image of $H_q(K)$ under the inclusion map inside $H_q(L)$ is exactly $\ker\left(\partial_q^K\right)/\mathrm{im}\left(\partial_{q+1}^{L,K}\right)$. Let $B_q^K$ be the matrix representation of $\partial_q^K$. Choose an orthonormal basis of $C_{q+1}^{L,K}$ and let $B_{q+1}^{L,K}$ be the corresponding matrix representation of $\partial_{q+1}^{L,K}$ in this basis. Then, by \Cref{thm:weighted-basis-persistent-Laplcacian} 
\begin{align*}
    \lap_q^{K,L}=& B_{q+1}^{L,K}\left( B_{q+1}^{L,K}\right)^\mathrm{T}\lc W_q^K\rc^{-1}+W_q^K\left( B_q^K\right)^\mathrm{T}\lc W_{q-1}^K\rc^{-1} B_q^K\\
    =&\lc W_q^K\rc^{\frac{1}{2}}\lc W_q^K\rc^{-\frac{1}{2}}B_{q+1}^{L,K}\left(\lc W_q^K\rc^{-\frac{1}{2}} B_{q+1}^{L,K}\right)^\mathrm{T}\lc W_q^K\rc^{-\frac{1}{2}} \\
    +&\lc W_q^K\rc^{\frac{1}{2}}\left(\lc W_{q-1}^K\rc^{-\frac{1}{2}} B_q^K\lc W_q^K\rc^{\frac{1}{2}}\right)^\mathrm{T}\lc W_{q-1}^K\rc^{-\frac{1}{2}} B_q^K\lc W_q^K\rc^{\frac{1}{2}}\lc W_q^K\rc^{-\frac{1}{2}}.
\end{align*}

Let $A\coloneqq\lc W_{q-1}^K\rc^{-\frac{1}{2}} B_q^K\lc W_q^K\rc^{\frac{1}{2}}$ and $B\coloneqq \lc W_q^K\rc^{-\frac{1}{2}}B_{q+1}^{L,K}$. Then,
\begin{enumerate}
    \item $AB=\lc W_{q-1}^K\rc^{-\frac{1}{2}} B_q^K\lc W_q^K\rc^{\frac{1}{2}} \lc W_q^K\rc^{-\frac{1}{2}}B_{q+1}^{L,K}=\lc W_{q-1}^K\rc^{-\frac{1}{2}} B_q^KB_{q+1}^{L,K}=0.$
    \item $\lap_q^{K,L}=\lc W_q^K\rc^{\frac{1}{2}}\lc BB^\T+A^\T A \rc\lc W_q^K\rc^{-\frac{1}{2}}$
\end{enumerate}
Since both $W_{q-1}^K$ and $W_q^K$ are non-singular, we have that $\ker(A)\cong\ker\left(B_q^K\right) $, $\mathrm{im}(B)\cong\mathrm{im}\left(B_{q+1}^{L,K}\right)$ and $\ker\lc \lap_q^{K,L}\rc\cong\ker\left(BB^\T+A^\T A\right) $. It then follows from \Cref{claim:kernel of bb+aa} that 
\[\beta_q^{K,L}=\dim\lc \ker\left(B_q^K\right)/\mathrm{im}\left(B_{q+1}^{L,K}\right)\rc=\dim \lc\ker\lc \lap_q^{K,L}\rc\rc=\mathrm{nullity}\lc \Delta_q^{K,L}\rc.\]
\end{proof}

\begin{proof}[Proof of \Cref{lm:image in subspace}]
Consider $\pi^\perp\circ\partial_{q+1}^L:C_{q+1}^L\rightarrow \lc C_q^K\rc^\perp$ where $\pi^\perp:C_q^L\rightarrow\lc C_q^K\rc^\perp$ is the orthogonal projection. Then, $D_{q+1}^L$ is the matrix representation of $\pi^\perp\circ\partial_{q+1}^L$ and $C_{q+1}^{L,K}=\ker\lc\pi^\perp\circ\partial_{q+1}^L\rc$. So $R_{q+1}^L=D_{q+1}^LY$ is the matrix representation of $\pi^\perp\circ\partial_{q+1}^L$ after a change of basis of $C_{q+1}^L$. 
\begin{enumerate}
    \item If $I=\emptyset$, then since $R_{q+1}^L$ is column reduced, $R_{q+1}^L$ has full column rank. This implies that $\pi^\perp\circ\partial_{q+1}^L:C_{q+1}^L\rightarrow \lc C_q^K\rc^\perp$ is injective and thus $C_{q+1}^{L,K}=\ker\lc\pi^\perp\circ\partial_{q+1}^L\rc=\{0\}$.
    \item If $I\neq\emptyset$, then the column space of $Z=Y(:,I)$ coincides with $\ker\lc\pi^\perp\circ\partial_{q+1}^L\rc=C_{q+1}^{L,K}$. Since $Y$ is non-singular, $Z$ has full column rank. Therefore, the columns of $Z$ constitute a basis of $C_{q+1}^{L,K}.$
\end{enumerate}

Obviously, $B_{q+1}^L\lc[n_q^K],:\rc$ is the matrix representation of $\pi\circ\partial_{q+1}^L:C_{q+1}^L\rightarrow C_q^K$ where $\pi:C_q^L\rightarrow C_q^K$ is the orthogonal projection. Therefore, $\lc B_{q+1}^LY\rc\lc[n_q^K],:\rc= B_{q+1}^L\lc[n_q^K],:\rc Y$ is the matrix representation of $\pi\circ\partial_{q+1}^L$ under the new basis $Y$ of $C_{q+1}^L$. Now, assume that $I\neq \emptyset$. Since the column space of $Z=Y(:,I)$ is $C_{q+1}^{L,K}$, we have that $B_{q+1}^{L,K}=\lc B_{q+1}^LY\rc\lc[n_q^K],I\rc$ is the matrix representation of $\pi\circ\partial_{q+1}^L|_{C_{q+1}^{L,K}}=\partial_{q+1}^{L,K}$.
\end{proof}
\begin{proof}[Proof of \Cref{lm:psd-proper}]
Since $P$ is positive semi-definite, there exists a square matrix $E\in\R^{n\times n}$ such that $P = EE^\T$. Assume that $D$ has size $d\times d$. Let $E_1\coloneqq E([n-d],:)$ and let $E_2\coloneqq E([n]\backslash[n-d],:)$. Then, 
\[M = W^{-1}PW=\begin{pmatrix}W_1^{-1}E_1E_1^\T W_1& W_1^{-1}E_1E_2^\T W_2\\W_2^{-1}E_2E_1^\T W_1&W_2^{-1}E_2E_2^\T W_2\end{pmatrix}=\begin{pmatrix}W_1^{-1}A W_1& W_1^{-1}B W_2\\W_2^{-1}C W_1&W_2^{-1}D W_2\end{pmatrix}\]
Since both $W_1$ and $W_2$ are non-singular, it is easy to see that $\ker(W_2^{-1}DW_2)=\ker\lc E_2^\T W_2\rc\subseteq\ker(W_1^{-1}BW_2)$ and $\ker\lc (W_2^{-1}DW_2)^\T\rc=\ker\lc E_2^\T W_2^{-1}\rc\subseteq\ker\lc (W_2^{-1}CW_1)^\T\rc$. Therefore, $W_2^{-1}DW_2$ is proper in $M$.

Note that 
\begin{align*}
    W_1^{-1}(P/D)W_1&=W_1^{-1}\lc A -BD^\dagger C \rc W_1\\
    &=W_1^{-1}\lc E_1E_1^\T -E_1E_2^\T \lc E_2E_2^\T \rc^\dagger E_2E_1^\T \rc W_1 
\end{align*}
and
\begin{align*}
    M/(W_2^{-1}DW_2)&=W_1^{-1}E_1E_1^\T W_1-W_1^{-1}E_1E_2^\T W_2\lc W_2^{-1}E_2E_2^\T W_2\rc^\dagger W_2^{-1}E_2E_1^\T W_1\\
    &=W_1^{-1}\lc E_1E_1^\T -E_1E_2^\T W_2\lc W_2^{-1}E_2E_2^\T W_2\rc^\dagger W_2^{-1} E_2E_1^\T \rc W_1
\end{align*}
To prove that $M/(W_2^{-1}DW_2)=W_1^{-1}(P/D)W_1$, we then only need to show that 
\[E_2^\T \lc E_2E_2^\T \rc^\dagger E_2=E_2^\T W_2\lc W_2^{-1}E_2E_2^\T W_2\rc^\dagger W_2^{-1} E_2.\]
To this end, we need the following elementary facts from linear algebra and interested readers are referred to \cite{barata2012moore} for a proof:
\begin{claim}\label{claim:range kernel}
Fix any $n\times m$ real matrix $H$. Then,
$H H^\dagger=\pi_{\mathrm{im}(H)}=\mathbb{I}_n-\pi_{\ker(H^\T)}$, where $\mathbb{I}_n$ is the $n$-dim identity matrix, and for any subspace $S\subseteq \R^n$, $\pi_S\in\R^{n\times n}$ denotes the orthogonal projector onto  $S$. Similarly, $H^\dagger H=\pi_{\mathrm{im}(H^\T)}=\mathbb{I}_m-\pi_{\ker(H)}$.
\end{claim}

Therefore,
\begin{align*}
    E_2^\dagger E_2\cdot E_2^\T W_2\lc W_2^{-1}E_2E_2^\T W_2\rc^\dagger W_2^{-1} E_2&=E_2^\dagger W_2\lc W_2^{-1} E_2 E_2^\T W_2\rc\lc W_2^{-1}E_2E_2^\T W_2\rc^\dagger W_2^{-1} E_2\\
    &=E_2^\dagger W_2\lc \mathbb{I}_{d}-\pi_{\ker(W_2E_2E_2^\T W_2^{-1})}\rc W_2^{-1} E_2\\
    &=E_2^\dagger W_2\lc \mathbb{I}_{d}-\pi_{\ker(E_2^\T W_2^{-1})}\rc \lc E_2^\T W_2^{-1}\rc^\T\\
    &=E_2^\dagger W_2 \lc E_2^\T W_2^{-1}\rc^\T= E_2^\dagger E_2.
\end{align*}
Similarly, we have that 
\[E_2^\dagger E_2\cdot E_2^\T \lc E_2E_2^\T \rc^\dagger E_2=E_2^\dagger E_2.\]
Note that $E_2^\dagger E_2=\pi_{\mathrm{im}(E_2^\T)}$. Then, $E_2^\dagger E_2\cdot E_2^\T=E_2^\T$ and thus
\[E_2^\T \lc E_2E_2^\T \rc^\dagger E_2=E_2^\dagger E_2=E_2^\T W_2\lc W_2^{-1}E_2E_2^\T W_2\rc^\dagger W_2^{-1} E_2.\]
This concludes the proof.
\end{proof}

\begin{proof}[Proof of \Cref{lm:eigen-interlacing}]
Note that $M=W^{-1}PW$ is similar to $P$, $M/(W_2^{-1}DW_2)=W_1^{-1}(P/D)W_1$ is similar to $P/D$ (cf. \Cref{lm:psd-proper}) and $W_1^{-1}AW_1$ is similar to $A$. Then, $\lambda_k(M)=\lambda_k(P)$, $\lambda_k(M/(W_2^{-1}DW_2))=\lambda_k(P/D)$ and $\lambda_k(W_1^{-1}AW_1)=\lambda_k(A)$ for any $1\leq k\leq n-d$. Therefore, we only need to consider the case when $W=\mathbb{I}_n$ is the identity matrix. 

Now, $M=P$. If $D$ is non-singular, then $\ker(D)=0$ and thus $D$ is obviously proper in $M$. A proof of the interlacing property of the case can be found in {\cite[Theorem 3.1]{fan2002schur}}. 

Now, we assume that $D$ is singular. Let $D=\sum_{i=1}^d\lambda_i\varphi_i\varphi_i^\T$ be the eigen-decomposition of $D$ where $0\leq \lambda_1\leq\ldots\leq\lambda_d$ are eigenvalues and $\varphi_i$s are their corresponding eigenvectors in $\mathbb{R}^d$. Assume that $0=\lambda_1=\ldots=\lambda_\ell$ are all the 0 eigenvalues of $D$. For each $\eps>0$, define $D_\eps\coloneqq\eps\cdot\sum_{i=1}^\ell\varphi_i\varphi_i^\T+\sum_{i=\ell+1}^d\lambda_i\varphi_i\varphi_i^\T$. Then, $D_\eps$ is positive definite and in particular, non-singular. Define $M_\eps\coloneqq\begin{pmatrix}
A &  B\\
C & D_\eps
\end{pmatrix}$, which is still positive semi-definite. Then, since $D_\eps$ is non-singular, we have that
\[\lambda_k(M_\eps)\leq \lambda_k(M_\eps/D_\eps)\leq \lambda_k(A),\quad \forall 1\leq k\leq n-d.\]
Note that $D^\dagger=\sum_{i=\ell+1}^d\frac{1}{\lambda_i}\varphi_i\varphi_i^\T$. Then,
\[D_\eps^{-1}=\frac{1}{\eps}\cdot\sum_{i=1}^\ell\varphi_i\varphi_i^\T+\sum_{i=\ell+1}^d\frac{1}{\lambda_i}\varphi_i\varphi_i^\T=\frac{1}{\eps}\cdot\sum_{i=1}^\ell\varphi_i\varphi_i^\T+D^\dagger.\]
For each $i=1,\ldots,\ell$, $D\varphi_i=0$. Since $D$ is proper, then $B\varphi_i=0$ for each $i=1,\ldots,\ell$. Then, $BD_\eps^{-1}=BD^\dagger$ and thus $BD_\eps^{-1}C=BD^\dagger C$. This implies that $M/D=M_\eps/D_\eps$. Since $M_\eps$ converges to $M$ as $\eps\rightarrow 0$, by continuity of eigenvalues, we have that
\[\lambda_k(M)\leq \lambda_k(M/D)\leq \lambda_k(A),\quad \forall 1\leq k\leq n-d.\]
\end{proof}
\begin{proof}[Proof of \Cref{lm:schur-complement-matrix-operation}]
We first assume that $B_2$ has full column rank. Then, $\rank(M)\leq\rank(B)=\rank(B_2)$. By \Cref{lm:schur-complement-rank}, we have that
\begin{align*}
    &\rank(B_2)\geq \rank(M)\geq\rank(M_{22})+\rank(M/M_{22})\\
    =&\rank(B_2B_2^\mathrm{T}W_2)+\rank(M/M_{22})=\rank(B_2)+\rank(M/M_{22}).
\end{align*}
Therefore, $\rank(M/M_{22})=0$ and thus $M/M_{22}=0$.

Now, we assume that $\rank(B_2)<m$. We let $X\coloneqq B_1Y_1\lc Y_1^\mathrm{T}Y_1\rc^{-1}(B_1Y_1)^\mathrm{T}W_1$. We first assume that $Y$ is an orthonormal matrix. Then, $X=B_1Y_1Y_1^\mathrm{T}B_1^\T W_1. $

Now, we compute $M$ in an alternative way:
\begin{align*}
    M=BB^\T W=BY Y^\T B^\T W=\begin{pmatrix}
B_1Y_1Y_1^\T B_1^\T W_1+ B_1Y_2Y_2^\T B_1^\T W_1& B_1Y_2Y_2^\T  B_2^\T W_2\\
B_2Y_2Y_2^\T B_1^\T W_1 & B_2Y_2Y_2^\T B_2^\T W_2
\end{pmatrix}
\end{align*}
where we have used that $B_2Y_1=0$. Consider 
\[\begin{pmatrix}
B_1Y_2Y_2^\T B_1^\T W_1& B_1Y_2Y_2^\T B_2^\T W_2\\
B_2Y_2Y_2^\T B_1^\T W_1 & B_2Y_2Y_2^\T B_2^\T W_2
\end{pmatrix}=\begin{pmatrix}
B_1Y_2\\
B_2Y_2
\end{pmatrix}\begin{pmatrix}
B_1Y_2\\
B_2Y_2
\end{pmatrix}^\T \begin{pmatrix}
W_1 & 0\\
0 & W_2
\end{pmatrix}. \]
Since $B_1Y_2$ is of full column rank, by \Cref{lm:schur-complement-rank} again we have that
\begin{align*}
    &\rank(B_1Y_2)\geq \rank \begin{pmatrix}
B_1Y_2Y_2^\T B_1^\T W_1& B_1Y_2Y_2^\T B_2^\T W_2\\
B_2Y_2Y_2^\T B_1^\T W_1 & B_2Y_2Y_2^\T B_2^\T W_2
\end{pmatrix}\\
\geq &\rank\lc B_1Y_2Y_2^\T B_1^\T W_1\rc + \rank \left(B_1Y_2Y_2^\T B_1^\T W_1 - B_1Y_2Y_2^\T B_2^\T W_2\lc B_2Y_2Y_2^\T B_2^\T W_2\rc^\dagger  B_2Y_2Y_2^\T B_1^\T W_1\right)\\
=&\rank(B_1Y_2) + \rank \left(B_1Y_2Y_2^\T B_1^\T W_1 - B_1Y_2Y_2^\T B_2^\T W_2\lc B_2Y_2Y_2^\T B_2^\T W_2\rc^\dagger  B_2Y_2Y_2^\T B_1^\T W_1\right).
\end{align*}
This implies that 
\[B_1Y_2Y_2^\T B_1^\T W_1 - B_1Y_2Y_2^\T B_2^\T W_2\lc B_2Y_2Y_2^\T B_2^\T W_2\rc^\dagger  B_2Y_2Y_2^\T B_1^\T W_1=0. \]
Therefore,
\begin{align*}
    M/M_{22}& = M_{11}-M_{12}M_{22}^\dagger  M_{21}\\
    &=X+B_1Y_2Y_2^\T B_1^\T W_1 - B_1Y_2Y_2^\T B_2^\T W_2\lc B_2Y_2Y_2^\T B_2^\T W_2\rc^\dagger  B_2Y_2Y_2^\T B_1^\T W_1=X.
\end{align*}

Now, suppose $Y$ is not orthonormal, then consider the QR factorization of $Y$: $Y=QR$ where $Q$ is an $m\times m$ orthonormal matrix and $R$ is an $m\times m$ non-sigular upper-triangular matrix. Suppose $Y_1$ has size $m\times \ell$. Write $Q$ and $R$ as block matrices as follows:
\[Q=\begin{pmatrix}
Q_1 & Q_2\end{pmatrix}\text{ and }R = \begin{pmatrix}
R_{11} & R_{12}\\
R_{21} & R_{22}
\end{pmatrix},\]
where $Q_1\in\R^{m\times \ell}$ and $R_{11}\in\R^{\ell\times \ell}$. Then, both $R_{11}$ and $R_{22}$ are non-singular and $R_{21}$ is a zero matrix. Then, by $Y=QR$ we have that $Y_1=Q_1R_{11}$ and $Y_2=Q_1R_{12}+Q_2R_{22}$. This implies that $B_2Q_1=B_2Y_1R_{11}^{-1}=0$ { and thus  }$B_2Q_2=B_2Y_2R_{22}^{-1}$ { has full column rank.}
Moreover,
\begin{align*}
    X&= B_1Y_1\lc Y_1^\mathrm{T}Y_1\rc^{-1}(B_1Y_1)^\mathrm{T}W_1=B_1Q_1R_{11}\lc \lc Q_1R_{11}\rc^\mathrm{T}Q_1R_{11}\rc^{-1}(B_1Q_1R_{11})^\mathrm{T}W_1\\
    &=B_1Q_1R_{11}\lc  R_{11}^\mathrm{T}Q_1^\T Q_1R_{11}\rc^{-1}R_{11}^\T(B_1Q_1)^\mathrm{T}W_1\\
    &=B_1Q_1R_{11}R_{11}^{-1}\lc  Q_1^\T Q_1\rc^{-1}\lc R_{11}^\mathrm{T}\rc^{-1} R_{11}^\T(B_1Q_1)^\mathrm{T}W_1\\
    &=B_1Q_1\lc Q_1^\mathrm{T}Q_1\rc^{-1}(B_1Q_1)^\mathrm{T}W_1
\end{align*}
Then, to prove that $X=M/M_{22}$, we can first reduce to the case when $Y$ is orthonormal and this concludes the proof.
\end{proof}

\begin{proof}[Proof of \Cref{prop:cheeger constant ineq}]
We only prove that $h^{K,L_1}\leq h^{K,L_2}$ and the inequality $h^{K}=h^{K,K}\leq h^{K,L_1}$ follows directly by replacing $K$ with $L_1$ and $L_1$ with $L_2$.

The following argument is adapted from the proof of \Cref{thm:monotonicity up eigen}. Since $C_{1}^{L_1,K}\subseteq C_{1}^{L_2,K}$, we consider an orthogonal decomposition $C_{1}^{L_2,K}=C_{1}^{L_1,K}\bigoplus \left(C_{1}^{L_1,K}\right)^\perp$. Then, we have the decomposition $\partial_{1}^{L_2,K}=\partial_{1}^{L_1,K}\oplus \partial^\perp$, where $\partial^\perp$ maps $\left(C_{1}^{L_1,K}\right)^\perp$ into $C_0^{K}$.
Therefore, we have that
\begin{align*}
    \Delta_{0,\mathrm{up}}^{K,L_2}={\partial_{1}^{L_2,K} \left(\partial_{1}^{L_2,K}\right)^*}={\partial_{1}^{L_1,K} \left(\partial_{1}^{L_1,K}\right)^*}+\partial^\perp \left(\partial^\perp\right)^*=\Delta_{0,\mathrm{up}}^{K,L_1}+\partial^\perp \left(\partial^\perp\right)^*.
\end{align*}
In this way, $\Delta_{0,\mathrm{up}}^{K,L_1}\leq \Delta_{0,\mathrm{up}}^{K,L_2}$ in the sense that for each $c\in C_0^K$,
\[\left\langle \Delta_{0,\mathrm{up}}^{K,L_1}c,c\right\rangle_{w_0^K}\leq \left\langle \Delta_{0,\mathrm{up}}^{K,L_2}c,c\right\rangle_{w_0^K}.\]

Therefore, for any nonempty $A\subsetneq V_K$, we have that 
\begin{align*}
    \mathfrak{C}^{L_1}_{A,V^K\backslash A} &= \lc\chi_A^K\rc^\T\lap_0^{K,L_1}\chi_A^K= \lc\chi_A^K\rc^\T\lap_{0,\mathrm{up}}^{K,L_1}\chi_A^K\\
    &\leq\lc\chi_A^K\rc^\T\lap_{0,\mathrm{up}}^{K,L_2}\chi_A^K= \lc\chi_A^K\rc^\T\lap_0^{K,L_2}\chi_A^K=\mathfrak{C}^{L_2}_{A,V^K\backslash A}.
\end{align*}
Hence, we conclude that $h^{K,L_1}\leq h^{K,L_2}$.
\end{proof}


\section{Computation of matrix representations of up and down Laplacians}\label{sec:computation of up and down}

When $K=(V^K,E^K,w^K)$ is a weighted graph such that $w^K_0\equiv 1$, we have that
\begin{equation}\label{eq:weighted graph laplacian}
    \lap_0^{K}={B_{1}^{K}W_{1}^K \left(B_{1}^{K}\right)^\T}=D_{0}^K-A_{0}^K,
\end{equation}
where $D_0^K$ is the diagonal degree matrix and $A_0^K$ is the adjacency matrix; more specifically, for any $i,j\in [n_0^K]$,
\begin{enumerate}
\item $D_{0}^K(i,i)=\sum_{\{v_j:\,\{v_i,v_j\}\in E^K\}}w_1^K(\{v_i,v_j\})$;
    \item $A_{0}^K(i,j)=\delta_{\{v_i,v_j\}\in E^K}\cdot w^K_1(\{v_i,v_j\})$.
\end{enumerate}
An analogous formula also holds for higher dimensional up and down Laplacians, which we review next.

Let $K$ be a simplicial complex with a weight function $w^K$ and fix $q\in\mathbb{N}$. All simplices are assumed to be arbitrarily oriented. Let the sets of oriented simplices $\Bar{S}_{q-1}^K=\{[\rho_i]\}_{i=1}^{n_{q-1}^K}$, $\Bar{S}_q^K=\{[\sigma_i]\}_{i=1}^{n_q^K}$ and $\Bar{S}_{q+1}^K=\{[\tau_i]\}_{i=1}^{n_{q+1}^K}$ be bases of $C_{q-1}^K,C_q^K$ and $C_{q+1}^K$, respectively. 

We define two $n_q^K\times n_q^K$ diagonal degree matrices $D_{q,\mathrm{up}}^K$ and $D_{q,\mathrm{down}}^K$ as follows: for any $i\in[n_q^K]$ 
\[D_{q,\mathrm{up}}^K(i,i)\coloneqq\sum_{\substack{{\tau_j\in S_{q+1}^K}\\{\sigma_i\text{ is a face of }\tau_j}}}\frac{w_{q+1}^K(\tau_j)}{w_q^K(\sigma_i)},\]
\[D_{q,\mathrm{down}}^K(i,i)\coloneqq\sum_{\substack{{\rho_j\in S_{q-1}^K}\\{\rho_j\text{ is a face of }\sigma_i}}}\frac{w_q^K(\sigma_i)}{w_{q-1}^K(\rho_j)}.\]
We further define two $n_q^K\times n_q^K$ adjacency matrices $A_{q,\mathrm{up}}^K$ and $A_{q,\mathrm{down}}^K$ as follows: for any $i,j\in[n_q^K]$ 
\[ A_{q,\mathrm{up}}^K(i,j)\coloneqq-\frac{w_{q+1}^K(\sigma_i\cup\sigma_j)}{w_q^K(\sigma_j)}\cdot[\sigma_i\cup\sigma_j:\sigma_i]\cdot[\sigma_i\cup\sigma_j:\sigma_j],\]
\[ A_{q,\mathrm{down}}^K(i,j)\coloneqq-\frac{w_q^K(\sigma_i)}{w_{q-1}^K(\sigma_i\cap\sigma_j)}\cdot[\sigma_i:\sigma_i\cap\sigma_j]\cdot[\sigma_j:\sigma_i\cap\sigma_j].\]
Here $[\sigma_i\cup\sigma_j:\sigma_i]$ denotes the sign of $[\sigma_i]$ in $\partial_{q+1}^K ([\sigma_i\cup\sigma_j])$ if $\sigma_i\cup\sigma_j\in S_{q+1}^K$ and is 0 otherwise. 
Similarly, $[\sigma_i:\sigma_i\cap\sigma_j]$ denotes the sign of $[\sigma_i\cap\sigma_j]$ in $\partial_{q}^K [\sigma_i]$ if $\sigma_i\cap\sigma_j\in S_{q-1}^K$ and is 0 otherwise.

Then, it is not hard to see that
\[\lap_{q,\mathrm{up}}^K = D_{q,\mathrm{up}}^K-A_{q,\mathrm{up}}^K\text{ and  }\lap_{q,\mathrm{down}}^K = D_{q,\mathrm{down}}^K-A_{q,\mathrm{down}}^K.\]
See \cite[Section 3.3]{goldberg2002combinatorial} or \cite{horak2013spectra} for more details.

\paragraph{{Computation of $\lap_{q,\mathrm{up}}^K$ and $\lap_{q,\mathrm{down}}^K$} and complexity analysis} 
{Given the boundary matrices $B_{q+1}^K$ and $B_q^K$, and the weight matrices $W_{q+1}^K,W_q^K$ and $W_{q-1}^K$, we describe how we construct the degree matrices $D_{q,\mathrm{up}}^K,D_{q,\mathrm{down}}^K$ and the adjacency matrices $A_{q,\mathrm{up}}^K,A_{q,\mathrm{down}}^K$, and give the time complexity of our constructions. }
\begin{enumerate}
    \item $D_{q,\mathrm{up}}^K$: We start with a $n_q^K\times n_q^K$ zero matrix $D_{q,\mathrm{up}}^K$.
    Next, we scan over each $\tau_j\in S_{q+1}^K$ and update $D_{q,\mathrm{up}}^K$ by adding $\frac{W_{q+1}^K(j,j)}{W_q^K(i,i)}$ to $D_{q,\mathrm{up}}^K(i,i)$ if $\sigma_i$ is a face of $\tau_j$ {(i.e., $B_{q+1}^K(i,j)\neq 0$)}. Since each $\tau_j$ has $q+2$ faces, it takes $O\lc(q+2)n_{q+1}^K\rc = O\lc n_{q+1}^K\rc$ total time to construct $D_{q,\mathrm{up}}^K$. 
    \item $D_{q,\mathrm{down}}^K$: Again we start with a $n_q^K\times n_q^K$ zero matrix $D_{q,\mathrm{down}}^K$.
    Next, we scan over each $\sigma_j\in S_{q}^K$ and update $D_{q,\mathrm{down}}^K$ by adding $\frac{W_{q}^K(j,j)}{W_{q-1}^K(i,i)}$ to $D_{q,\mathrm{down}}^K(i,i)$ if $\rho_i$ is a face of $\sigma_j$ {(i.e., $B_{q}^K(i,j)\neq 0$)}. Since each $\sigma_j$ has $q+1$ faces, it takes $O\lc(q+1)n_{q}^K\rc = O\lc n_{q}^K\rc$ total time to construct $D_{q,\mathrm{down}}^K$.
    \item $A_{q,\mathrm{up}}^K$: First, note that any two $q$-simplices $\sigma_i$ and $\sigma_j$ can only both be faces of at most one ($q+1$)-simplex. Now for each ($q+1$)-simplex $\tau_\ell\in S_{q+1}^K$, we need to enumerate any two co-dimension 1 faces $\sigma_i$ and $\sigma_j$ of $\tau_\ell$, and fill in the entry $A_{q,\mathrm{up}}^K(i,j)$ by $-\frac{W_{q+1}^K(\ell,\ell)}{W_q^K(j,j)}B_{q+1}^K(i,\ell)B_{q+1}^K(j,\ell)$. This takes $O\lc\begin{pmatrix}q+2\\2
    \end{pmatrix}\cdot n_{q+1}^K\rc = O\lc n_{q+1}^K\rc$ total time. 
    \item $A_{q,\mathrm{down}}^K$: For any $\sigma_i,\sigma_j\in S_q^K$, there exists at most one $(q-1)$-simplex $\rho_\ell\in S_{q-1}^K$ as the common face of $\sigma_i$ and $\sigma_j$. Then, for each pair $\sigma_i,\sigma_j\in S_q^K$, if they have no common face then $A_{q,\mathrm{down}}^K(i,j)=0$ and if they have a common face $\rho_\ell\in S_{q-1}^K$, then fill in the entry $A_{q,\mathrm{down}}^K(i,j)$ by $-\frac{W_{q}^K(i,i)}{W_{q-1}^K(\ell,\ell)}B_{q}^K(\ell,i)B_{q}^K(\ell,j)$. Since we have $O\lc\lc n_q^K\rc^2\rc$ many pairs of $\sigma_i,\sigma_j\in S_q^K$, the time complexity for obtaining $A_{q,\mathrm{down}}^K$ is $O\lc\lc n_q^K\rc^2\rc$.
    
\end{enumerate}
This implies that computing $\lap_{q,\mathrm{up}}^K$ takes time $O\left(n_{q+1}^K\right)$ and computing $\lap_{q,\mathrm{down}}^K$ takes time $O\left(\lc n_q^K\rc^2\right)$. 

\section{A remark for the matrix representation of $\lap_{q,\mathrm{up}}^{K,L}$ in \cite{wang2020persistent}}\label{sec:wrong algorithm example}
When dealing with unweighted simplicial complexes, i.e., $w^L\equiv1$, it is suggested in \cite{wang2020persistent} that $\lap_{q,\mathrm{up}}^{K,L}$ can be computed by (i) considering a certain submatrix of the boundary operator and then (ii) multiplying it by its transpose. However, as we show in \Cref{thm:weighted-basis-persistent-Laplcacian} and \Cref{lm:image in subspace}, finding the matrix representations of both the new boundary operator and its dual is much more involved {than what is suggested}: the boundary matrix $B_{q+1}^L$ has to be reduced, and the matrix representation of the dual operator $\lc\partial_{q+1}^{L,K}\rc^*$ is not simply the transpose $\lc B_{q+1}^{L,K}\rc^\T$ but instead has the form $\lc Z^\T Z\rc^{-1}\lc B_{q+1}^{L,K}\rc^\T$ (cf. \cref{lm:image in subspace}).
The following example illustrates that simply considering a certain submatrix of the boundary matrix (in a way suggested in \cite{wang2020persistent}) and then multiplying it by its transpose \emph{does not} produce the correct up persistent Laplacian, in the sense that persistent Betti number cannot be recovered.

\begin{example}\label{ex:wrong algorithm}
Consider the graph $L$ shown in \Cref{fig:4point} with vertices labeled as in the figure. Let $K$ be the subgraph with vertex set $V^K=\{1,2\}$. Choose orientations and an order of edges as follows: $\bar{S}_1^L=\{[1,3],[3,4],[4,2]\}$. Then,
$B_1^L=\begin{pmatrix}-1 &0 & 0\\0&0&1\\1&-1&0&\\0&1&-1\end{pmatrix}.$ It is suggested in \cite{wang2020persistent} to use the following matrix $B=B_1^L\lc\{1,2\},:\rc=\begin{pmatrix}-1 &0 & 0\\0&0&1\end{pmatrix}$ as the matrix representation of $\partial_1^{L,K}$. Then, $BB^\T=\begin{pmatrix}1 &0 \\0&1\end{pmatrix}=I_2$. Note that $\mathrm{nullity}\lc I_2\rc=0$. However, it is obvious that $\beta_0^{K,L}=1$. Then, $\beta_0^{K,L}\neq \mathrm{nullity}\lc BB^\T\rc$ and thus $BB^\T$ cannot be the correct matrix representation of the persistent Laplacian $\Delta_0^{K,L}$.
\end{example}
\begin{figure}
    \centering
    \includegraphics[width=0.25\linewidth]{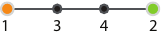}
    \caption{\textbf{Illustration of \cref{ex:wrong algorithm}.} A four-vertex graph.}
    \label{fig:4point}
\end{figure}

\section{Effective resistances between disjoint sets}\label{sec:effective resistance between sets}

In this section, we study some properties of effective resistances between disjoint sets.

\subsection{Energy formulation}\label{sec: energy formulation}

Let $K=(V^K,E^K,w^K)$ be a connected weighted graph. Let $f:V^K\rightarrow\mathbb{R}$ be any real function. We define its energy to be
\[\mathcal{E}^K(f)\coloneqq\frac{1}{2}\sum_{{v,w\in V^K;v\sim w}}w^K(\{v,w\})(f(v)-f(w))^2,\]
where $v\sim w$ denotes the condition that $\{v, w\}\in E^K$.

We prove that the effective resistance between disjoints sets of vertices is related to the energy functional as follows:
\begin{theorem}\label{thm:alternative definition of resist}
Let $A,B\subseteq V^K$ be nonempty disjoint subsets. Then,
\begin{equation}\label{eq:alternative resist}
    \mathfrak{R}^K_{A,B}=\lc\inf\left\{\mathcal{E}^K(f):\,f|_A\equiv1\text{ and }f|_B\equiv 0\right\}\rc^{-1}.
\end{equation}
\end{theorem}

In \cite{lyons2017probability,lyons2020induced}, the right hand side of \Cref{eq:alternative resist} is taken as the definition of effective resistance between sets. Therefore, the theorem illustrates that the definition of effective resistances we adopt in this paper (namely that of \cite{song2019extension})  is equivalent to the one from \cite{lyons2017probability,lyons2020induced}. 

In fact, \Cref{eq:alternative resist} is a proper generalization of the following property of effective resistances between vertices:
\begin{lemma}[{\cite[Theorem 4.2]{jorgensen2008operator}}]\label{lm:vertex energy}
For any distinct $v,w\in V^K$ we have that
\[\mathfrak{R}^K_{v,w}=\lc\inf\left\{\mathcal{E}^K(f):\,f(v)=1\text{ and }f(w)= 0\right\}\rc^{-1}.\]
\end{lemma}

It is suggested in \cite{song2019extension} that $\mathfrak{R}_{A,B}^K$ is the same as the effective resistance between two vertices in a reduced graph $\hat{K}$ by collapsing vertices both in $A$ and in $B$. More precisely, $\hat{K}$ is defined as follows: 
\begin{enumerate}
    \item $V^{\hat{K}}\coloneqq \{a,b\}\cup V^K\backslash (A\cup B)$ where $a,b$ are two extra vertices not belonging to $V^K$.
    \item For any $v,w\in V^{\hat{K}}\backslash\{a,b\}$, $\{v,w\}\in E^{\hat{K}}$ if and only if $\{v,w\}\in E^K$; for any $w\in V^{\hat{K}}\backslash\{a,b\}$, $\{a,w\}\in E^{\hat{K}}$ if there exists $v\in A$ such that $\{v,w\}\in E^K$; for any $v\in V^{\hat{K}}\backslash\{a,b\}$, $\{v,b\}\in E^{\hat{K}}$ if there exists $w\in B$ such that $\{v,w\}\in E^K$; $\{a,b\}\in E^{\hat{K}}$ if there exists $v\in A$ and $w\in B$ such that $\{v,w\}\in E^K$.
    \item For any $\{v,w\}\in E^{\hat{K}}$, if $v,w\in V^{\hat{K}}\backslash\{a,b\}$, then $w^{\hat{K}}(\{v,w\})\coloneqq w^{{K}}(\{v,w\})$; if $v=a$ and $w\in V^{\hat{K}}\backslash\{a,b\}$, then
    \[w^{\hat{K}}(\{v,w\})\coloneqq\sum_{v'\in A;v'\sim w}w^{{K}}(\{v',w\});\] if $v\in V^{\hat{K}}\backslash\{a,b\}$ and $w=b$, then
    \[w^{\hat{K}}(\{v,w\})\coloneqq\sum_{w'\in B;v\sim w'}w^{{K}}(\{v,w'\});\] if $v=a$ and $w=b$, then \[w^{\hat{K}}(\{v,w\})\coloneqq\sum_{v'\in A,w'\in B;v'\sim w'}w^{{K}}(\{v',w'\}).\]
\end{enumerate}

Then, it follows directly from \cite[Equation (10)]{song2019extension} that
\begin{lemma}\label{lm:AB to ab}
$\mathfrak{R}_{A,B}^K=\mathfrak{R}_{a,b}^{\hat{K}}$.
\end{lemma}

Note that the statement of the lemma emulates the ``physical" experiment depicted in Figure \ref{fig:res-sets}, namely that all vertices in each set are connected together by a perfect conductor (cable) and then one measures the ratio between voltage and current to obtain the value of the effective resistance between the two sets; see  \cite[Theorem 1]{song2019extension} for more details.

Now, we are ready to prove \Cref{thm:alternative definition of resist}.

\begin{proof}[Proof of \Cref{thm:alternative definition of resist}]
For any $f:V^K\rightarrow\mathbb{R}$ such that $f|_A\equiv 1$ and $f|_B\equiv 0$, we have that
\begin{align*}
    \mathcal{E}^K(f)=&\frac{1}{2}\sum_{v,w\in V^K;  v\sim w}w^K(\{v,w\})(f(v)-f(w))^2\\
    =&\frac{1}{2}\sum_{v,w\in V^K\backslash(A\cup B);  v\sim w}w^K(\{v,w\})(f(v)-f(w))^2 \\
    + &\sum_{v\in A,w\in V^K\backslash(A\cup B);  v\sim w}w^K(\{v,w\})(1-f(w))^2\\
    + &\sum_{v\in V^K\backslash(A\cup B),w\in B;  v\sim w}w^K(\{v,w\})(f(v)-0)^2\\
    + &\sum_{v\in A,w\in B;  v\sim w}w^K(\{v,w\})(1-0)^2
    \end{align*}

Now, we define a new function $\hat{f}:V^{\hat{K}}:\rightarrow\mathbb{R}$ as follows: for any $v\in V^{\hat{K}}\backslash\{a,b\}$, $\hat{f}(v)\coloneqq f(v)$, $\hat{f}(a)\coloneqq1$ and $\hat{f}(b)\coloneqq0$. Then, the following equalities are direct consequences of the definition of $\hat{K}$:
\[\sum_{v,w\in V^K\backslash(A\cup B);  v\sim w}w^K(\{v,w\})(f(v)-f(w))^2=\sum_{v,w\in V^{\hat{K}}\backslash\{a,b\};  v\sim w}w^{\hat{K}}(\{v,w\})(\hat{f}(v)-\hat{f}(w))^2;\]
\[\sum_{v\in A,w\in V^K\backslash(A\cup B);  v\sim w}w^K(\{v,w\})(1-f(w))^2=\sum_{v=a,w\in V^{\hat{K}}\backslash\{a,b\};  v\sim w}w^{\hat{K}}(\{a,w\})(1-\hat{f}(w))^2;\]
\[\sum_{v\in V^K\backslash(A\cup B),w\in B;  v\sim w}w^K(\{v,w\})(f(v)-0)^2=\sum_{v\in V^{\hat{K}}\backslash\{a,b\},w=b;  v\sim w}w^{\hat{K}}(\{v,b\})(\hat{f}(v)-0)^2;\]
\[\sum_{v\in A,w\in B;  v\sim w}w^K(\{v,w\})(1-0)^2=w^{\hat{K}}(\{a,b\})(1-0)^2.\]

Therefore,
\begin{align*}
    \mathcal{E}^K(f)=&\frac{1}{2}\sum_{v,w\in V^{\hat{K}}\backslash\{a,b\};  v\sim w}w^{\hat{K}}(\{v,w\})(\hat{f}(v)-\hat{f}(w))^2 \\
    + &\sum_{v=a,w\in V^{\hat{K}}\backslash\{a,b\};  v\sim w}w^{\hat{K}}(\{a,w\})(1-\hat{f}(w))^2\\
    + &\sum_{v\in V^{\hat{K}}\backslash\{a,b\},w=b;  v\sim w}w^{\hat{K}}(\{v,b\})(\hat{f}(v)-0)^2\\
    + &w^{\hat{K}}(\{a,b\})(1-0)^2\\
    =&\frac{1}{2}\sum_{v,w\in V^{\hat{K}};  v\sim w}w^{\hat{K}}(\{v,w\})(\hat{f}(v)-\hat{f}(w))^2\\
    =&\mathcal{E}^{\hat{K}}(\hat{f}).
\end{align*}

Similarly, for any $\hat{f}:V^{\hat{K}}:\rightarrow\mathbb{R}$ such that $\hat{f}(a)=1$ and $\hat{f}(b)=0$, the function $f:V^K\rightarrow\mathbb{R}$ defined by $f|_A\coloneqq 1$ and $f|_B\coloneqq 1$ and $f(v)\coloneqq \hat{f}(v)$ for any $v\in V^{{K}}\backslash(A\cup B)$ satisfies
\[\mathcal{E}^{\hat{K}}(\hat{f})=\mathcal{E}^K(f).\]

Therefore, by \Cref{lm:vertex energy} and \Cref{lm:AB to ab} we have that
\begin{align*}
    \mathfrak{R}^K_{A,B}&=\mathfrak{R}^{\hat{K}}_{a,b}=\lc\inf\left\{\mathcal{E}^{\hat{K}}(\hat{f}):\,\hat{f}(a)=1\text{ and }\hat{f}(w)= 0\right\}\rc^{-1}\\
    &=\lc\inf\left\{\mathcal{E}^K(f):\,f|_A\equiv1\text{ and }f|_B\equiv 0\right\}\rc^{-1}.
\end{align*}

\end{proof}

\subsection{Relation with random walks}\label{sec:relation random walk}
Let $K=(V^K,E^K,w^K)$ be a connected weighted graph. Consider a Markov chain $X_0,X_1,\ldots$ on $K$ with $V^K$ being the set of states and with $p\coloneqq \lc D_0^K\rc^{-1}A_0^K$ being the transition matrix, where $D_0^K$ is the degree matrix and $A_0^K$ is the adjacency matrix (cf. \Cref{sec:computation of up and down}). More explicitly, for $v,w\in V^K$, the transition probability is given by
\[p(v,w)\coloneqq\frac{w^K(\{v,w\})}{\deg^K(v)}.\]

It is not hard to prove that this Markov chain is irreducible and reversible, and has a unique stationary distribution 
\[\pi:=\sum_{v\in V^K}\frac{\deg^K(v)}{\sum_{w\in V^K}\deg^K(w)}\delta_v.\]

Given an initial distribution $\mu$ of $X_0$, we denote by $P_\mu$ the law of the Markov chain $\{X_n\}_{n=0}^\infty$. When $\mu=\delta_v$ is the Dirac delta measure at $v\in V^K$, we let $P_v\coloneqq P_{\delta_v}$.

\paragraph{Harmonic functions} A real function $f:V^K\rightarrow \mathbb{R}$ is said to be \textit{harmonic} at a vertex $v\in V^K$ if 
\[f(v)=\sum_{w\sim v} p(v,w)f(w)=\sum_{w\sim v} \frac{w^K(\{v,w\})}{\deg^K(v)}f(w).\]
Given $S\subseteq V^K$, we say $f$ is harmonic on $S$ if $f$ is harmonic at each vertex in $S$. The following two lemmas are easily adapted from \cite{doyle1984random}.

\begin{lemma}[Maximum principle]\label{lm:max prin}
Let $S\subsetneq V^K$ and let $f:V^K\rightarrow \mathbb R$. Suppose that $f$ is harmonic on $S$. Then, $f$ reaches its maximum and minimum on $V^K\backslash S$. 
\end{lemma}
\begin{proof}
We only prove the case for the maximum value, and the minimum case follows the same argument.
Let $M\coloneqq\max_{v\in V^K}f(v)$. Let $ V_M\coloneqq\{v\in V^K:\,f(v)=M\}$. If $V_M\cap S=\emptyset$, then we are done. Hence we assume that $V_M\cap S\neq\emptyset$. Note that for any $v\in V_M\cap S$ and any $w\in V^K$, if $\{v,w\}\in E^K$, then it is easy to see that $w\in V_M$ since $f$ is harmonic at $v$. 

Now, for any $w\in V_M\backslash S$, since $K$ is connected, there exists a path $v=v_0,v_1,\ldots,v_k=w$ such that $v\in V_M\cap S$ and for $i=0,\ldots,k-1$, $\{v_i,v_{i+1}\}\in E^K$. Assume that $i\in1,\ldots,k$ is the smallest integer such that $v_i\notin S$. Such $i$ exists since $v_0\in S$ and $v_k\notin S$. Therefore, $v_0,v_1,\ldots,v_{i-1}\in S$ and $v_i\notin S$. Then, by inductively applying the argument in the previous paragraph, we have that $v_1,v_2,\ldots,v_{i-1}\in V_M$ and thus $v_i\in V_M\backslash S\subseteq V^K\backslash S$. This concludes the proof. 
\end{proof}

\begin{lemma}[Uniqueness principle]\label{lm:uniq prin}
If $f$ and $g$ are harmonic on $S$ and are such that $f|_{V^K\backslash S}=g|_{V^K\backslash S}$, then $f\equiv g$.
\end{lemma}
\begin{proof}
Let $h\coloneqq f-g$. Then, $h$ is obviously harmonic on $S$ and $h|_{V^K\backslash S}\equiv 0$. By \Cref{lm:max prin}, $h$ attains its maximum and minimum on $V^K\backslash S$. So $h\equiv 0$ and thus $f\equiv g$.
\end{proof}

\paragraph{Probabilistic interpretation of voltage functions} Given two disjoint subsets $A,B\subseteq V^K$, the \textit{voltage function} (from source $A$ to ground $B$) is the unique function $U:V^K\rightarrow\mathbb{R}$ which is harmonic on $V^K\backslash A\cup B$ and satisfies $U|_A\equiv1$ and $U|_B\equiv0$. Here the uniqueness follows from \Cref{lm:uniq prin}. 

When $A=\{a\}$ and $B=\{b\}$ are singletons, there exists a well-known probabilistic interpretation of the corresponding voltage function $U$. To illustrate this, we first introduce some notation. For any subspace $S\in V^K$, let $T_S^0\coloneqq\min\{n\geq 0:\,X_n\in S\}$ be the first time when the Markov chain visits $S$. Then,
\begin{proposition}[{\cite[Section 1.3.2]{doyle1984random}}]\label{prop:voltage interpretation}
For any $v\in V^K$, one has that
\[U(v) = P_v\lc T_a^0<T_b^0\rc.\]
\end{proposition}

The following generalization of Proposition \ref{prop:voltage interpretation} has been mentioned in passing in \cite[Section 3.1]{zhu2003semi}. We provide a proof for completeness.
\begin{theorem}\label{thm:prob potential}
For any $v\in V^K$, one has that
\[U(v) = P_v\lc T_A^0<T_B^0\rc.\]
\end{theorem}
\begin{proof}
Let $S\coloneqq V^K\backslash A\cup B$. Then, we have that $U$ is a harmonic function on $S$. Now, for any $v\in S$, we have that
\[P_v\lc T_A^0<T_B^0\rc = \sum_{w\sim v}p(v,w)P_w\lc T_A^0<T_B^0\rc.\]
If we define $h:V^K\rightarrow\mathbb{R}$ by $v\mapsto P_v\lc T_A^0<T_B^0\rc$, then $h$ is clearly a harmonic function on $S$ such that $h|_A\equiv 1$ and $h|_B\equiv 0$. By \Cref{lm:uniq prin}, we have that $h\equiv U$, which concludes the proof.
\end{proof}

\paragraph{Effective resistance and escape probability} For any subspace $S\in V^K$, let $T_S^1\coloneqq\min\{n\geq 1:\,X_n\in S\}$. It is obvious that for any $v\in V^K$ and any subsets $A,B\subseteq V^K$ we have 
\[P_v\lc T_A^1<T_B^1\rc=\sum_{w\sim v}p(v,w)P_w\lc T_A^0<T_B^0\rc.\]
Given $a,b\in V^K$, we call $P_a\lc T_b^1<T_a^1\rc$ the \textit{escape probability} from $a$ to $b$, i.e., the probability of the random walk, starting at $a$, reaches $b$ before returning to $a$. The escape probability is closely related with effective conductance:
\begin{proposition}[{\cite[Section 1.3.4]{doyle1984random}}]
$P_a\lc T_b^1<T_a^1\rc=\frac{\mathfrak{C}^K_{a,b}}{\deg^K(a)}$.
\end{proposition}

Now, we generalize this result to the case of two disjoint sets $A$ and $B$. Recall that $\pi$ denotes the stationary distribution. We call $P_\pi\lc X_0\in A,T_B^1<T_A^1 \rc$ \textit{the escape probability from $A$ to $B$}.\footnote{Here $P_\pi(\mathcal{A},\mathcal{B})$ denotes the probability of the intersection  $\mathcal{A}\cap \mathcal{B}$ of the two events $\mathcal{A}$ and $\mathcal{B}$.} Then, we have the following result:

\begin{theorem}\label{thm:escape set}
$P_\pi\lc X_0\in A,T_B^1<T_A^1 \rc=\frac{\mathfrak{C}^K_{A,B}}{\sum_{a\in A}\deg^K(a)}$.
\end{theorem}
\begin{proof}
Let $U$ be the voltage function on $K$ such that $U|_A\equiv 1$ and $U|_B\equiv 0$.  Assume that $V^K = \{v_1,\ldots,v_n\}$. Then, we overload the notation $U$ to also denote the vector $ \lc U(v_1),U(v_2),\ldots, U(v_n)\rc^\T$.
Let
\begin{equation}\label{eq:ohm law}
    J\coloneqq\lap_0^K U.
\end{equation}
For each vertex $v_i\in V^K$, $J(v_i)$ is actually the (influx of) electric current  at $v_i$ under the voltage function $U$.
For any given $a\in A$, it follows directly from \Cref{eq:ohm law} that
\begin{align*}
    J(a) &= \sum_{v\sim a}(U(a)-U(v))w^K(\{a,v\})\\
    &=\sum_{v\sim a}(U(a)-U(v))\frac{w^K(\{a,v\})}{\deg^K(a)}\deg^K(a)\\
    &=\lc 1 - \sum_{v\sim a}p(a,v)U(v)\rc\deg^K(a).
\end{align*}
By \Cref{thm:prob potential}, we have that 
\[\sum_{v\sim a}p(a,v)U(v)=\sum_{v\sim a}p(a,v)P_v\lc T_A^0<T_B^0\rc=P_a\lc T_A^1<T_B^1\rc.\]
Therefore, 
\[J(a)=\lc 1-P_a\lc T_A^1<T_B^1\rc\rc\deg^K(a)=P_a\lc T_A^1>T_B^1\rc\deg^K(a).\]
By \cite[Theorem 1]{song2019extension}, we have that $\sum_{a\in A}J(a)=\mathfrak{C}^K_{A,B}$. This implies that 
\begin{align*}
    \mathfrak{C}^K_{A,B}&=\sum_{a\in A}J(a)=\sum_{a\in A}P_a\lc T_A^1>T_B^1\rc\deg^K(a)\\
    &=\sum_{a\in A}\deg^K(a)\cdot \sum_{a'\in A}P_{a'}\lc T_A^1>T_B^1\rc\frac{\deg^K(a')}{\sum_{a\in A}\deg^K(a)}\\
    &=\sum_{a\in A}\deg^K(a)\cdot \sum_{a'\in A}P_{a'}\lc T_A^1>T_B^1\rc\pi(a')\\
    &=\sum_{a\in A}\deg^K(a)\cdot P_\pi\lc X_0\in A,T_A^1>T_B^1 \rc.
\end{align*}
Hence, $P_\pi\lc X_0\in A,T_B^1<T_A^1 \rc=\frac{\mathfrak{C}^K_{A,B}}{\sum_{a\in A}\deg^K(a)}$.

\end{proof}

\section{Discussion about the persistent Cheeger constant}\label{sec:alt-pcc}

Consider the graph shown in \Cref{fig:ngraph} which we call $L_n$ for $n\in\mathbb{N}$. Let $K$ be always the graph consisting of vertices $a$ and $b$ regardless of $n$. Then, it is obvious that 
\[h^{K,L_n}=\mathfrak{C}_{a,b}^{L_n}.\]

\begin{figure}
    \centering
    \includegraphics[width=0.5\linewidth]{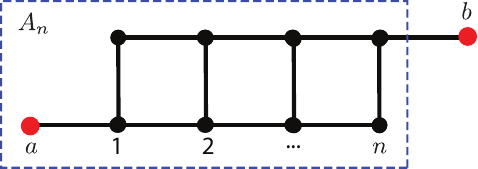}
    \caption{\textbf{Graph $L_n$.} $L_n$ is an unweighted graph with $2n+2$ vertices. $A_n$ is the complement of $b$ in the vertex set $V^{L_n}$.}
    \label{fig:ngraph}
\end{figure}

Let $A_n$ denote the vertex set $V^{L_n}\backslash\{b\}$. Then, by \Cref{thm:alternative definition of resist} and \Cref{lm:vertex energy}, we have that
\begin{align*}
\mathfrak{C}^{L_n}_{a,b}&=\inf\left\{\mathcal{E}^K(f):\,f(a)=1\text{ and }f(b)= 0\right\}\\
&\leq \inf\left\{\mathcal{E}^K(f):\,f|_{A_n}\equiv1\text{ and }f(b)= 0\right\}=
    \mathfrak{C}^{L_n}_{A_n,b}.
\end{align*}
On the other hand, using notation in \Cref{sec:effective resistance between sets}, by \Cref{lm:AB to ab} we obtain
\[\mathfrak{C}^{L_n}_{A_n,b}=\mathfrak{C}^{\hat{L}_n}_{a_n,b},\]
where $\hat{L}_n$ is the two-vertex graph with vertex set $\{a_n,b\}$ and with one unit weight edge connecting them. Therefore, for each $n\in\mathbb{N}$
\[\mathfrak{C}^{L_n}_{A_n,b}=\mathfrak{C}^{\hat{L}_n}_{a_n,b}=1.\]
Hence, for all $n\in\mathbb{N}$, $h^{K,L_n}=\mathfrak{C}_{a,b}^{L_n}$ is upper bounded by 1. However, it is obvious that $|P_{L_n}(a,b)|$ (and thus $h^{K,L_n}_\mathrm{path}$)  blows up to infinity as $n\rightarrow\infty$. As $\lambda_{0,2}^{K,L_n}\leq 2 h^{K,L_n}\leq 2$ by \Cref{thm:persistent Cheeger inequality}, $h^{K,L_n}_\mathrm{path}$ cannot be upper bounded by $\lambda_{0,2}^{K,L_n}$ for all $n\in\mathbb{N}$.

\section{Effective resistances for simplicial networks and Kron reduction}\label{sec:kron reduction}

For any positive $q_0\in\mathbb N$, a \emph{$q_0$-dim simplicial network} is a $q_0$-dim simplicial complex $K$ with a weight function $w^K$ such that $w_q^K\equiv 1$ for all $q\neq q_0$. This condition of weight functions follows from  \cite{kook2018simplicial}. In fact, in the one dimensional case, whereas edge weights represent electric conductance, there is no physical interpretation for weights on vertices.

In this section, we define the effective resistances for simplicial networks and study the corresponding properties.

\paragraph*{Effective resistances for simplicial networks} Given a positive $q_0\in\mathbb{N}$, let $K$ be a $q_0$-dim simplicial network. We represent explicitly the vertex set $S_0^K$ of $K$ by an ordered finite set $[n_0^K]=\left\{1,\ldots,n_0^K\right\}$. A $q_0$-dim \emph{{(electric)} current generator} $\sigma$ is a $(q_0+1)$-point subset of $[n_0^K]$ such that $\partial_{q_0}^K\sigma \in\mathrm{im}(\partial_{q_0}^K)$ \cite{kook2018simplicial}.
Here $\sigma$ may not be a simplex in $K$, and $\partial_{q_0}^K\sigma$ denotes the formal boundary of $\sigma$ computed via \Cref{eq:boundary map}.
Note that any $q_0$-simplex $\sigma$ in $K$ is automatically a $q_0$-dim current generator. Let $\partial_\sigma \coloneqq\partial_{q_0}^K\sigma $ and let $D_{\sigma}^K\in \R^{n_{q_0-1}^K}$ denote the vector representation of $\partial_\sigma \in C_{q_0-1}^K$. Then, we define the effective resistance $\mathfrak{R}_\sigma^K$ on a current generator $\sigma$ by 
\begin{equation}\label{eq:effective resistance simplicial}
    \mathfrak{R}_\sigma^K\coloneqq \lc D_{\sigma}^K\rc^\mathrm{T}\lc\lap_{q_0-1,\mathrm{up}}^K\rc^{\dagger}D_{\sigma}^K.
\end{equation}

\begin{remark}[Connection with graph effective resistance]
When $q_0=1$, if $v,w\in S_0^K$ belong to the same connected component of $K$, then $[v,w]$ is a $1$-dim current generator. It is clear that $\mathfrak{R}_{v,w}^K$ defined via \Cref{eq:effective resistance graph} coincides with $\mathfrak{R}_{[v,w]}^K$ as defined via \Cref{eq:effective resistance simplicial}. Therefore, our definition of effective resistances on current generators is a generalization of effective resistances between vertices on graphs.
\end{remark}

\begin{remark}
In \cite{osting2017towards}, a formula similar to \Cref{eq:effective resistance simplicial} has been used to define effective resistances on $q_0$-simplices of a $q_0$-dim simplicial network. Note that our setting is more general in that we define effective resistances on all $q_0$-dim current generators, not just the set of $q_0$-simplices.
\end{remark}

We define the \emph{generalized current-balance equation} for a $q_0$-dim simplicial network $K$ as follows:
\begin{equation}\label{eq:current-balance equation}
    J = \lap_{q_0-1,\mathrm{up}}^K U,
\end{equation}
where $J,U\in \R^{n_{q_0-1}^K}$ are the vector representations of chains in $ C_{q_0-1}^K$ reflecting current influxes and voltage potentials at $(q_0-1)$-simplices, respectively.

\begin{lemma}[Effective resistance and current-balance equation]\label{lm:effective resistance current balance}
Let $\sigma$ be a $q_0$-dim current generator and let $j_\sigma\in \R$. Let $U\coloneqq j_\sigma\lc \lap_{q_0-1,\mathrm{up}}^K\rc^\dagger D_\sigma^K\in\R^{n_{q_0-1}^K}$ representing a chain in $C_{q_0-1}^K$. Then, $j_\sigma D_{\sigma}^K$ and $U$ satisfy the current-balance equation (\Cref{eq:current-balance equation}):
\begin{equation}\label{eq:i=du}
    j_\sigma D_{\sigma}^K=\lap_{q_0-1,\mathrm{up}}^K U.
\end{equation}
Moreover, if $j_\sigma\neq 0$, then we have
\[\mathfrak{R}_\sigma^K= \frac{\lc D_{\sigma}^K\rc^\mathrm{T}U}{j_\sigma}.\]
\end{lemma}
\begin{proof}
We need the following property of current generators.
\begin{claim}\label{clm:current generator perp}
If $\sigma$ is a current generator, then $\partial_\sigma \perp\ker\lc \Delta_{q_0-1,\mathrm{up}}^K\rc$.
\end{claim}
\begin{proof}[Proof of \Cref{clm:current generator perp}]
Since $\sigma$ is a current generator, there exists a chain $c_\sigma\in C_{q_0}^K$ such that $\partial_\sigma =\partial_{q_0}^K c_\sigma$. It is obvious that $\ker\lc\Delta_{q_0-1,\mathrm{up}}^K\rc=\ker\lc \lc\partial_{q_0}^K\rc^*\rc$. Then, for any $c\in\ker(\Delta_{q_0-1,\mathrm{up}}^K)$, we have that
\begin{align*}
    \left\langle c,\partial_\sigma \right\rangle_{w_{q_0-1}^K}=\left\langle c,\partial_{q_0}^Kc_\sigma\right\rangle_{w_{q_0-1}^K}=\left\langle \lc \partial_{q_0}^K\rc^*c,c_\sigma\right\rangle_{w_{q_0}^K}=0.
\end{align*}
This implies that $\partial_\sigma \perp\ker\left(\Delta_{q_0-1,\mathrm{up}}^K\right)$.
\end{proof}
Then, $\lc \lap_{q_0-1,\mathrm{up}}^K\rc^\dagger \lap_{q_0-1,\mathrm{up}}^K D_\sigma^K=D_\sigma^K $. Therefore,
\begin{align*}
    \lap_{q_0-1,\mathrm{up}}^K U= j_\sigma \lap_{q_0-1,\mathrm{up}}^K\lc\lap_{q_0-1,\mathrm{up}}^K\rc^\dagger D_\sigma^K=j_\sigma\lc \lap_{q_0-1,\mathrm{up}}^K\rc^\dagger\lap_{q_0-1,\mathrm{up}}^K D_\sigma^K=j_\sigma D_\sigma^K.
\end{align*}
If $j_\sigma\neq 0$, then
\begin{align*}
    \mathfrak{R}_\sigma^K= \lc D_{\sigma}^K\rc^\mathrm{T}\lc\lap_{q_0-1,\mathrm{up}}^K\rc^{\dagger}D_{\sigma}^K=\frac{\lc D_{\sigma}^K\rc^\mathrm{T}U}{j_\sigma}.
\end{align*}
\end{proof}

\subparagraph{Relation with the notion of effective resistance defined in
\cite{kook2018simplicial}} Let $q_0$ be a positive integer. Given a $q_0$-dim simplicial network $K$, a version of effective resistance $\tilde{\mathfrak{R}}_\sigma^K$ on a current generator $\sigma$ is defined in \cite{kook2018simplicial} differently from \Cref{eq:effective resistance simplicial}. In particular, $\tilde{\mathfrak{R}}_\sigma^K$ is characterized in \cite{kook2018simplicial} via the following formula:
\begin{theorem}[{\cite[Theorem 4.2]{kook2018simplicial}}]
Let $K$ be a $q_0$-dim simplicial network and let $\sigma$ be a $q_0$-dim current generator. If the $(q_0-1)$-th reduced homology\footnote{The $q$-th reduced homology of a simplicial complex $K$ is the $q$-th homology group of the extended chain complex $\cdots\xrightarrow{\partial_{q+1}^K} C_q^K\xrightarrow{\partial_{q}^K}\cdots\xrightarrow{\partial_{1}^K}C_0^K\xrightarrow{\tilde{\partial}_{0}^K}\R$, where $\tilde{\partial}_0^K$ is a linear map sending each vertex $v\in S_0^K$ to $1\in\R$.} $\tilde{H}_{q_0-1}(K)=0$, then $\lap_{q_0-1}^K$ is non-singular and 
\[\tilde{\mathfrak{R}}_\sigma^K=\lc D_{\sigma}^K\rc^\mathrm{T}\lc\lap_{q_0-1}^K\rc^{-1}D_{\sigma}^K.\]
\end{theorem}

It turns out that $\mathfrak{R}_\sigma^K=\tilde{\mathfrak{R}}_\sigma^K$ when $\tilde{H}_{q_0-1}(K)= 0$:
\begin{theorem}\label{thm:two effective resistance same}
Let $K$ be a $q_0$-dim simplicial network and let $\sigma$ be a $q_0$-dim current generator. Then,
\begin{equation}\label{eq:two definition of resistance}
    {R}_\sigma^K=\lc D_{\sigma}^K\rc^\mathrm{T}\lc\lap_{q_0-1,\mathrm{up}}^K\rc^{\dagger}D_{\sigma}^K=\lc D_{\sigma}^K\rc^\mathrm{T}\lc\lap_{q_0-1}^K\rc^{\dagger}D_{\sigma}^K.
\end{equation}
In particular, when $\tilde{H}_{q_0-1}(K)= 0$, we have that $\mathfrak{R}_\sigma^K=\tilde{\mathfrak{R}}_\sigma^K$.
\end{theorem}

The proof of this theorem is based on the following result about the relation between the generalized inverse of Laplacians and the generalized inverses of up and down Laplacians.

\begin{lemma}\label{lm:invers L = Lu+Ld}
Let $K$ be a simplicial complex with a weight function $w^K$. If $w_q^L\equiv 1$ for a given positive $q\in\mathbb N$,  then
\[\lc\lap_{q}^K\rc^\dagger=\lc\lap_{q,\mathrm{up}}^K\rc^{\dagger}+\lc\lap_{q,\mathrm{down}}^K\rc^{\dagger}. \]
\end{lemma}
\begin{proof}
By \Cref{rmk:w=1 weighted laplacian} we have that $\lap_{q}^K$, $\lap_{q,\mathrm{up}}^K$ and $\lap_{q,\mathrm{down}}^K$ are symmetric positive semi-definite matrices. Then, consider the eigen-decompositions $\lap_{q,\mathrm{up}}^K=\sum_{i}\lambda_i\phi_i\phi_i^\mathrm{T}$ and $\lap_{q,\mathrm{down}}^K=\sum_{j}\mu_j\psi_j\psi_j^\mathrm{T}$ where $\lambda_i,\mu_j\neq 0$. Since $\mathrm{im}\lc\lap_{q,\mathrm{up}}^K\rc\subseteq\ker\lc\lap_{q,\mathrm{down}}^K\rc$ and $\mathrm{im}\lc\lap_{q,\mathrm{down}}^K\rc\subseteq\ker\lc\lap_{q,\mathrm{up}}^K\rc$ (see \cite[Theorem 2.2]{horak2013spectra}), we then have the following eigen-decomposition of $\lap_{q}^K$:
\[\lap_{q}^K=\sum_i \lambda_i\phi_i\phi_i^\mathrm{T}+\sum_{j}\mu_j\psi_j\psi_j^\mathrm{T}. \]
Therefore, 
\[\lc\lap_{q}^K\rc^\dagger=\sum_i \lambda_i^{-1}\phi_i\phi_i^\mathrm{T}+\sum_{j}\mu_j^{-1}\psi_j\psi_j^\mathrm{T}=\left(\lap_{q,\mathrm{up}}^K\right)^{\dagger}+\left(\lap_{q,\mathrm{down}}^K\right)^{\dagger}. \]
\end{proof}

\begin{proof}[Proof of \Cref{thm:two effective resistance same}]
When $q_0=1$, $\lap_{q_0-1}^K=\lap_0^K=\lap_{0,\mathrm{up}}^K=\lap_{q_0-1,\mathrm{up}}^K$. Then, \Cref{eq:two definition of resistance} holds trivially.

Now, we assume that $q_0>1$. Since $w_{q_0-1}^K\equiv 1$, by \Cref{lm:invers L = Lu+Ld}, we only need to show that 
\[\lc D_{\sigma}^K\rc^\T (\lap_{q_0-1,\mathrm{down}}^K)^{\dagger}D_{\sigma}^K =0.\] 
Since $\sigma$ is a current generator, there exists a chain $c_\sigma\in C_{q_0}(K)$ such that $\partial_\sigma  =\partial_{q_0}c_\sigma$. Consider the eigen-decomposition $\lap_{q_0-1,\mathrm{down}}^K=\sum_{j}\mu_j\psi_j\psi_j^\mathrm{T}$ where $\mu_j\neq 0$. Each $\psi_j\in \R^{n_{q_0-1}^K}$ represents a chain in $C_{q_0-1}^K$, which we still denote by $\psi_j$. Then,
\begin{align*}
    \psi_j^\mathrm{T}D_{\sigma}^K &=\left\langle \mu_j^{-1}\Delta_{q_0-1}^K\psi_j,\partial_{q_0}^Kc_\sigma\right\rangle_{w_{q_0-1}^K}=\left\langle \mu_j^{-1}\lc\partial_{q_0-1}^K\rc^*\partial_{q_0-1}^K\psi_j,\partial_{q_0}^Kc_\sigma\right\rangle_{w_{q_0-1}^K}\\
    &=\left\langle \mu_j^{-1}\partial_{q_0-1}^K\psi_j,\partial_{q_0-1}^K\partial_{q_0}^Kc_\sigma\right\rangle_{w_{q_0-2}^K}=\left\langle \mu_j^{-1}\partial_{q_0-1}^K\psi_j,0\right\rangle_{w_{q_0-2}^K}=0.
\end{align*}
Therefore, $\lc D_{\sigma}^K\rc^\T \left(\lap_{q_0-1,\mathrm{down}}^K\right)^{\dagger}D_{\sigma}^K =0$.
\end{proof}

\paragraph*{Relationship between the up persistent Laplacian and the effective resistance}
Let $K\hookrightarrow L$ be a simplicial pair. For simplicity of presentation, we assume for each $q\in\mathbb N$ an ordering $\bar{S}_q^L=\{[\sigma_i]\}_{i=1}^{n_q^L}$ on $\Bar{S}_q^L$ such that $\bar{S}_q^K=\{[\sigma_i]\}_{i=1}^{n_q^K}$. The main goal is to prove the following result stating that the up persistent Laplacian preserves the effective resistances for simplicial networks.
\begin{theorem}\label{thm:high-dim-effective-persistent-preserve}
Let $K\hookrightarrow L$ be a simplicial pair. Let $q_0$ be a positive integer and suppose that $L$ is a $q_0$-dim simplicial network. Let $\sigma$ be a $q_0$-dim current generator in $L$. If $\partial_{\sigma}=\partial_{q_0}^L\sigma\in C_{q_0-1}^K$, then
\[\mathfrak{R}_{\sigma}^L=\lc D_{\sigma}^L\rc^{\T}\lc\lap_{q_0-1,\mathrm{up}}^L\rc^\dagger D_{\sigma}^L=\lc D_{\sigma}^K\rc^{\T}\lc\lap_{q_0-1,\mathrm{up}}^{K,L}\rc^\dagger D_{\sigma}^K,\]
where $D_\sigma^L\in \R^{n_{q_0-1}^L}$ and $D_{\sigma}^K\in \R^{n_{q_0-1}^K}$ denote the vector representations of $\partial_{\sigma}$ in $C_{q_0-1}^L$ and $C_{q_0-1}^K$, respectively.
\end{theorem}

To prove the theorem, we need the following auxiliary result.

\begin{lemma}
Let $J,U\in \R^{n_{q_0-1}^L}$ be two vectors satisfying \Cref{eq:current-balance equation} for the simplicial network $L$. Let $J_K\coloneqq J\lc[n_{q_0-1}^K]\rc, J_{L\backslash K}\coloneqq J\lc I_K^L\rc$ and let $U_K\coloneqq U\lc[n_{q_0-1}^K]\rc, U_{L\backslash K}\coloneqq U\lc I_K^L\rc$, where $I_K^L=[n_{q_0-1}^L]\backslash[n_{q_0-1}^K]$.  Then,
\begin{equation}\label{eq:kron-linear-equation}
    J_K - \lap_{q_0-1,\mathrm{up}}^L\lc [n_{q_0-1}^K],I_{K}^L\rc \lc\lap_{q_0-1,\mathrm{up}}^L\lc I_K^L,I_K^L\rc\rc^\dagger J_{L\backslash K}=\lap_{q_0-1,\mathrm{up}}^{K,L} U_K.
\end{equation}
In particular, if we regard $S_{q_0-1}^L\backslash S_{q_0-1}^K$ (indexed by $I_K^L$) as ``interior nodes'' of the simplicial network $L$,  we let $J_{L\backslash K}=0$ and then the current influxes at ``boundary nodes'' in $S_{q_0-1}^K$ are completely determined by voltage potentials on $S_{q_0-1}^K$ and the up persistent Laplacian:
\begin{equation}\label{eq:I_K=kron U_K}
  J_K= \lap_{q_0-1,\mathrm{up}}^{K,L} U_K.
\end{equation}
\end{lemma}

\begin{proof}
For notationally simplicity, we use abbreviations 
\begin{align*}
    \lap &\coloneqq \lap_{q_0-1,\mathrm{up}}^L,~~ \lap_{KK}\coloneqq \lap\lc[n_{q_0-1}^K],[n_{q_0-1}^K]\rc, ~~ \lap_{KL}\coloneqq\lap\lc[n_{q_0-1}^K],I_K^L\rc, \\
& ~~~~~~~~\lap_{LK} \coloneqq\lap\lc I_K^L,[n_{q_0-1}^K]\rc ~\text{and}~\lap_{LL}\coloneqq\lap\lc I_K^L,I_K^L\rc. 
    \end{align*}
Then,
\[\begin{pmatrix}J_K\\ J_{L\backslash K}\end{pmatrix}=\begin{pmatrix}\lap_{KK} & \lap_{KL}\\\lap_{LK}&\lap_{LL}\end{pmatrix}\begin{pmatrix}U_K\\ U_{L\backslash K}\end{pmatrix}.\]
Therefore, we have that 
\begin{align}
    &J_K=\lap_{KK}U_K+\lap_{KL}U_{L\backslash K}\text{ and}\label{eq:I_K}\\
    &J_{L\backslash K}=\lap_{LK}U_K+\lap_{LL}U_{L\backslash K}.\label{eq:I_L}
\end{align}
Left multiply \Cref{eq:I_L} by $\lap_{KL}\lap_{LL}^\dagger$ and obtain \[\lap_{KL}\lap_{LL}^\dagger J_{L\backslash K}=\lap_{KL}\lap_{LL}^\dagger\lap_{LK}U_K+\lap_{KL}\lap_{LL}^\dagger\lap_{LL}U_{L\backslash K}.\]
By \Cref{lm:psd-proper} we have that $\ker\lc\lap_{LL}\rc\subseteq \ker\lc\lap_{KL}\rc$. This is equivalent to the condition $\lap_{KL}=\lap_{KL}\lap_{LL}^\dagger\lap_{LL}$. Therefore, 
\begin{equation}\label{eq:I_L product}
    \lap_{KL}\lap_{LL}^\dagger J_{L\backslash K}=\lap_{KL}\lap_{LL}^\dagger\lap_{LK}U_K+\lap_{KL}U_{L\backslash K}
\end{equation}
Then, we obtain \Cref{eq:kron-linear-equation} by \Cref{thm:persis-Laplacian-schur-formula} and by subtracting \Cref{eq:I_L product} from \Cref{eq:I_K}.
\end{proof}

\begin{proof}[Proof of \Cref{thm:high-dim-effective-persistent-preserve}]
Let $U\coloneqq \lc \lap_{q_0-1,\mathrm{up}}^L\rc^\dagger D_\sigma^L\in \R^{n_{q_0-1}^L}$. Then, by \Cref{eq:i=du} we have that 
\[D_\sigma^L = \lap_{q_0-1,\mathrm{up}}^{L}U.\]
Hence $J\coloneqq D_\sigma^L$ and $U$ satisfy \Cref{eq:current-balance equation}. Furthermore, note that $D_\sigma^L\lc  [n_{q_0-1}^K]\rc=D_\sigma^K$ and $D_\sigma^L\lc [n_{q_0-1}^L]\backslash[n_{q_0-1}^K]\rc=0$. Then, by \Cref{eq:kron-linear-equation}, we have that
\[D_\sigma^K = \lap_{q_0-1,\mathrm{up}}^{K,L}U_K,\]
where $U_K=U\lc [n_{q_0-1}^K]\rc$.

Note that $\lc \lap_{q_0-1,\mathrm{up}}^{K,L} \rc^\dagger \lap_{q_0-1,\mathrm{up}}^{K,L} = \mathbb{I}_{q_0-1}-\pi_{\ker\lc \lap_{q_0-1,\mathrm{up}}^{K,L}\rc}$, where $\mathbb{I}_{q_0-1}$ is the $(q_0-1)$-dim identity matrix and $\pi_{\ker\lc \lap_{q_0-1,\mathrm{up}}^{K,L}\rc}:C_{q_0-1}^K\rightarrow C_{q_0-1}^K$ is the orthogonal projector such that $\mathrm{im}\lc\pi_{\ker\lc \lap_{q_0-1,\mathrm{up}}^{K,L}\rc}\rc= \ker\lc \lap_{q_0-1,\mathrm{up}}^{K,L}\rc$. Let $I\coloneqq \mathbb{I}_{q_0-1}$ and let $\pi\coloneqq \pi_{\ker\lc \lap_{q_0-1,\mathrm{up}}^{K,L}\rc}$. Then, 
\[\lc \lap_{q_0-1,\mathrm{up}}^{K,L} \rc^\dagger D_\sigma^K = (I-\pi)U_K.\]
Therefore, by \Cref{lm:effective resistance current balance}, we have that
\[\mathfrak{R}_\sigma^L={\lc D_\sigma^L\rc^\T U}={\lc D_\sigma^K\rc^\T U_K}={\lc D_\sigma^K\rc^\T \lc(I-\pi)U_K\rc}=\lc D_\sigma^K\rc^\T\lc\lap_{q_0-1,\mathrm{up}}^{K,L}\rc^\dagger D_\sigma^K,\]
where in the third equality we used the fact $\partial_\sigma\perp \ker\lc \Delta_{q_0-1,\mathrm{up}}^{K,L}\rc$ whose proof is essentially the same as the one for \Cref{clm:current generator perp}.
\end{proof}

\begin{remark}
When $K\hookrightarrow L$ is a weighted graph pair and $L$ is connected, if we let $\sigma=[v,w]$ for distinct vertices $v,w\in V^K$, then \Cref{thm:high-dim-effective-persistent-preserve} reduces to \Cref{thm:effective-persistent-preserve} and \cite[Theorem 3.8]{dorfler2012kron}.
\end{remark}

\paragraph*{Kron reduction for simplicial networks}Inspired by \Cref{thm:high-dim-effective-persistent-preserve}, it is tempting to generalize the Kron reduction of graphs to the case of simplicial networks by defining the up persistent Laplacian $\lap_{q_0-1,\mathrm{up}}^{K,L}$ as the \emph{simplicial Kron-reduced matrix}. However, there is no result analogous to \Cref{prop:persistent graph is a graph} for simplicial networks, namely, in general there exists no well-defined simplicial network with $\lap_{q_0-1,\mathrm{up}}^{K,L}$ being its $(q_0-1)$-th up Laplacian. Before providing a counterexample, we present a necessary condition for a matrix to be the $(q_0-1)$-th up Laplacian of a $q_0$-dim simplicial network.

\begin{proposition}\label{prop:necessary up lap}
Let $K$ be a $q_0$-dim simplicial network. Then, for any $i\in[n_{q_0-1}^K]$, we have that
\[\sum_{j\neq i}\big|\lap_{q_0-1,\mathrm{up}}^{K}(i,j)\big|=q_0\cdot\lap_{q_0-1,\mathrm{up}}^{K}(i,i).\]
\end{proposition}
\begin{proof}
From \Cref{sec:computation of up and down} we have that
\[\lap_{q_0-1,\mathrm{up}}^{K}(i,i)=\sum_{\substack{{\tau_k\in S_{q_0}^K}\\{\sigma_i\text{ is a face of }\tau_k}}}{w_{q_0}^K(\tau_k)}\]
and
\[\left|\lap_{q_0-1,\mathrm{up}}^{K}(i,j)\right|={w_{q_0}^K(\sigma_i\cup\sigma_j)}\cdot\delta_{\sigma_i\cup\sigma_j\in S_{q_0}^K}.\]
Therefore,
\begin{align*}
    \sum_{j\neq i}\big|\lap_{q_0-1,\mathrm{up}}^{K}(i,j)\big|&=\sum_{\substack{{\tau_k\in S_{q_0}^K}\\{\sigma_i\text{ is a face of }\tau_k}}}\sum_{\sigma_i\neq\sigma_j\text{ is a face of }\tau_k}\big|\lap_{q_0-1,\mathrm{up}}^{K}(i,j)\big|\\
    &=\sum_{\substack{{\tau_k\in S_{q_0}^K}\\{\sigma_i\text{ is a face of }\tau_k}}}\sum_{\sigma_i\neq\sigma_j\text{ is a face of }\tau_k}w_{q_0}^K(\tau_k)\\
    &=\sum_{\substack{{\tau_k\in S_{q_0}^K}\\{\sigma_i\text{ is a face of }\tau_k}}}q_0\cdot w_{q_0}^K(\tau_k)\\
    &=q_0\cdot\lap_{q_0-1,\mathrm{up}}^{K}(i,i),
\end{align*}
where in the third equality we used the fact that each $q_0$-dim simplex has $q_0+1$ faces.
\end{proof}

In \Cref{ex:simplicial} we construct an example of \emph{unweighted} simplicial pair $K\hookrightarrow L$ such that $\lap_{q_0-1,\mathrm{up}}^{K,L}$ violates \Cref{prop:necessary up lap}. Note that \Cref{prop:necessary up lap} holds due to rigidity of simplices, i.e., the number of faces of a simplex is determined by the dimension of the simplex. Such restriction can be eliminated if we consider more general cell complexes such as CW complexes. We leave for future work to generalize the Kron reduction in the context of certain cell complex networks.
\begin{figure}
    \centering
    \includegraphics[width=0.25\linewidth]{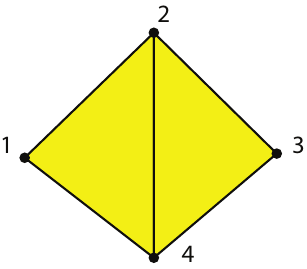}
    \caption{\textbf{Illustration of \cref{ex:simplicial}.} A simplicial complex.}
    \label{fig:simplicial}
\end{figure}
\begin{example}\label{ex:simplicial}
Consider the simplicial complex $L$ shown in \Cref{fig:simplicial} and assume that $w^L\equiv1 $. Let $K$ be the subcomplex $\{[1],[2],[3],[4],[1,2],[2,3],[3,4],[1,4]\}$.
Let $q_0=2$. Given the order $[1,2,3,4]$, it is easy to compute that 
 \[\lap_{1,\mathrm{up}}^{K,L}=\frac{1}{2}\begin{pmatrix}1&1&1&-1\\1&1&1&-1\\1&1&1&-1\\-1&-1&-1&1\end{pmatrix}.\]
 Therefore, for each $i=1,2,3,4$, we have that
 \[\sum_{j\neq i}\big|\lap_{1,\mathrm{up}}^{K}(i,j)\big|=\frac{3}{2}=3\cdot\lap_{1,\mathrm{up}}^{K}(i,i)\neq q_0\cdot\lap_{1,\mathrm{up}}^{K}(i,i)\]
 and thus there exists no well-defined simplicial network with $\lap_{1,\mathrm{up}}^{K,L}$ being its {$1$-st} up Laplacian.
\end{example}

\end{document}